\documentclass[11pt]{article}
\usepackage{}
\usepackage{amsfonts}
\usepackage{amssymb}
\usepackage{amsthm}
\usepackage{amsmath}
\usepackage{graphicx}
\usepackage{empheq}
\usepackage{indentfirst}
\usepackage{cite}
\usepackage{mathrsfs}
\usepackage{graphics}
\usepackage{enumerate}
\usepackage{color}

 \newtheorem{theorem}{Theorem}[section]
 \newtheorem{corollary}{Corollary}[section]
 \newtheorem{lemma}{Lemma}[section]
 \newtheorem{proposition}{Proposition}[section]

 \newtheorem{remark}{Remark}[section]
 \numberwithin{equation}{section}

\allowdisplaybreaks[4]

\def\e{\epsilon}

\def\k{\kappa}
\newcommand{\beq}{\begin{equation}}
\newcommand{\eeq}{\end{equation}}
 \def\non{\nonumber }
\def\bea{\begin{eqnarray}}
\def\eea{\end{eqnarray}}

\def\pa{\partial}
\def\na{\nabla}

\begin{document}
\title{Blow-up for a Three Dimensional Keller--Segel Model with Consumption of Chemoattractant}
\author{Jie Jiang\thanks{Wuhan Institute of Physics and Mathematics, Chinese Academy of Sciences,
Wuhan 430071, HuBei Province, P.R. China,
\textsl{jiang@wipm.ac.cn}.},\ \ \
Hao Wu\thanks{School of Mathematical Sciences and Shanghai Key Laboratory for Contemporary Applied Mathematics, Fudan University, Shanghai, China,
\textsl{haowufd@yahoo.com}.}\ \
and \ Songmu Zheng\thanks{School of Mathematical Sciences, Fudan University, Shanghai, China,
\textsl{songmuzheng@yahoo.com}.}}

\date{\today}
\maketitle

\begin{abstract}
We investigate blow-up properties for the initial-boundary value problem of a Keller--Segel model with consumption of chemoattractant when the spatial dimension is three. Through a kinetic reformulation of the Keller--Segel model, we first derive some higher-order estimates and obtain certain blow-up criteria for the local classical solutions. These blow-up criteria generalize the results in \cite{CKL13, CKL14} from the whole space $\mathbb{R}^3$ to the case of bounded smooth domain $\Omega\subset \mathbb{R}^3$. Lower global blow-up estimate on $\|n\|_{L^\infty(\Omega)}$ is also obtained based on our higher-order estimates.
Moreover, we prove local non-degeneracy for blow-up points.

\end{abstract}

\section{Introduction}
In this paper, we investigate the following chemotaxis system of Keller--Segel type with consumption of chemoattractant
\beq
 \begin{cases}\label{chemo1}
n_t= \Delta n-\nabla\cdot(n\chi(c)\nabla c),\\
c_t= \Delta c- nf(c),\end{cases}
\eeq
subject to the homogeneous Neumann conditions
\begin{equation}\label{chemo0}
	\frac{\partial n}{\partial \nu}\Big{|}_{\partial\Omega}
=\frac{\partial c}{\partial \nu}\Big{|}_{\partial\Omega}=0,
\end{equation}
and initial conditions
\begin{equation}\label{ini}
	n(x,0)=n_0(x),\quad c(x,0)=c_0(x).
\end{equation}
Here, $\Omega\subset\mathbb{R}^3$ is a bounded domain with smooth boundary $\partial\Omega$ and $\nu$ is the outward normal vector on $\partial \Omega$.
The unknowns $n=n(x,t)$ and $c=c(x,t)$ denote the bacteria density and chemoattractant concentration, respectively. $\chi(c)>0$ is a parameter that describes the chemotactic sensitivity and $f(c)$ is the consumption rate of the chemoattractant by the bacteria.
For the sake of simplicity, throughout this paper we assume that
\begin{align}
 \chi(c)=\chi>0,\quad f(c)=c,\label{assmp}
\end{align}
where $\chi$ is a positive constant.

The chemotaxis system \eqref{chemo1} can be regarded as a fluid-free version of the following  Keller--Segel--Navier--Stokes model, which was proposed by Tuval et al \cite{Tuval} in order to study the motion of oxygen-driven swimming bacteria in viscous incompressible fluids:
\beq\begin{cases}\label{chemofluid}
n_t+\mathbf{u}\cdot\nabla n=\Delta n-\nabla\cdot(n\chi(c)\nabla c),\\
c_t+\mathbf{u}\cdot\nabla c=\Delta c-nf(c),\\
\mathbf{u}_t+\kappa(\mathbf{u}\cdot\nabla)\mathbf{u}+\nabla p=\Delta \mathbf{u}-n\nabla\phi,\\
\nabla\cdot u=0.\end{cases}
\eeq
In \eqref{chemofluid}, $n, c$ represent the densities of bacteria and oxygen (chemoattractant), respectively and $\mathbf{u}$ stands for the velocity field of the macroscopic incompressible fluid subject to the Navier--Stokes equations, the scalar function $p$ stands for the pressure and $\phi$ is a given potential function that accounts for the effects of external forces such as gravity.
The chemotaxis-fluid system \eqref{chemofluid} has been extensively studied in the literature. We refer to \cite{LL2011,DLM2010,CKL13,CKL14,TZJMAA} for studies on the Cauchy problem in the whole space $\mathbb{R}^d$, and to \cite{FLM2010,LL2011,TW2013} for the case that the linear diffusion term $\Delta n$ is replaced by a nonlinear one of porous medium type $\Delta n^m$.
 Concerning the initial boundary value problem of system \eqref{chemofluid} in a bounded regular domain $\Omega\subset \mathbb{R}^d$, Lorz \cite{Lorz2010} proved the existence of a local weak solution by Schauder's fixed point theory when $d=3$. If the domain $\Omega$ is further assumed to be {\em convex}, using some delicate entropy-energy estimates, Winkler \cite{WinklerCPDE} established the existence of a unique global classical solution with large initial data for $\kappa\in\mathbb{R}$ when $d=2$, and the existence of a global weak solution for  $\kappa=0$ when $d=3$. Later in \cite{WinkARMA}, the same author proved that the global classical solution obtained in \cite{WinklerCPDE} in two dimensional case will converge to a constant state $(n_\infty, 0,\mathbf{0})$ as time goes to infinity. More recently, by exploiting an elementary lemma due to Mizoguchi \& Souplet \cite{MizoSoup}, the authors of the present paper derived a new type of entropy-energy estimate that holds on general bounded regular domains in $\mathbb{R}^3$ (see \cite{JWZ}) and thus generalized the previous work of Winkler \cite{WinkARMA,WinklerCPDE} in which the domain $\Omega$ was essentially assumed to be convex. For further results, we refer to \cite{WinklerCVPDE, ZL15a,TWDCDS2012,Winkler16} and the references therein.

In order to have a better understanding on the dynamics of chemotaxis in the coupled system \eqref{chemofluid}, it will be helpful to investigate its fluid-free version. Several works in this direction have been carried out in recent years.
For instance, taking $\mathbf{u}=\mathbf{0}$, $f(c)=c$ and $\chi(s)=\chi$ with $\chi>0$ being a constant, \eqref{chemofluid} is reduced to the chemotaxis model \eqref{chemo1} of Keller--Segel type. When $\Omega\subset \mathbb{R}^2$ is a bounded convex domain, global existence of classical solutions to the initial boundary value problem \eqref{chemo1}--\eqref{ini} with large non-negative initial data was established in \cite{BBTW}, while when $\Omega\subset\mathbb{R}^3$ is a bounded convex domain, Tao \& Winkler \cite{TW} obtained the existence of a generalized weak solution that enjoys eventual smoothness, i.e., there exists a $T>0$ such that $(n,c)$ is bounded and smooth in $\Omega\times (T,\infty)$.  On the other hand, Tao \cite{Tao} proved that on any bounded regular domain $\Omega\subset\mathbb{R}^d$ with $d\geq2$, problem \eqref{chemo1}--\eqref{ini} admits a unique global classical solution provided that the constant chemotactic sensitivity satisfies $0< \chi\leq \frac{1}{6(d+1)\|c_0\|_{L^\infty(\Omega)}}$. In other words, $\|c_0\|_{L^\infty(\Omega)}$ is required to be small for fixed $\chi$ in order to ensure global well-posedness. Here we note that $L^\infty(0,T;L^\infty(\Omega))$ is a critical scaling-invariant space for the unknown $c(x,t)$ as a solution to the system \eqref{chemo1}. Exponential convergence of global classical solutions to constant equilibrium $(\frac{1}{|\Omega|}\int_\Omega n_0dx,0)$ was recently examined in \cite{ZL15}. There are also some contributions devoted to the study on existence of classical solutions to \eqref{chemo1}--\eqref{ini} with nonlinear diffusion $\Delta n^m$ instead of the linear diffusion $\Delta n$ (see \cite{WMLZ15} and the references cited therein).

In this paper, we consider system \eqref{chemo1} in an arbitrary bounded regular domain $\Omega\subset\mathbb{R}^d$ without imposing any assumption on it convexity. When $d=2$, our previous work \cite{JWZ} yields that problem \eqref{chemo1}--\eqref{ini} admits a unique global classical solution without any restriction on the sizes of $\chi$ and $\|c_0\|_{L^\infty(\Omega)}$, which improves the result of \cite{Tao} in two dimensional case.
However, when $d=3$, to the best of our knowledge, it is still an open problem whether the classical solution to problem \eqref{chemo1}--\eqref{ini} exists globally or blows up in finite time with arbitrary large regular initial data.
As a first step towards this problem, it would be worth establishing some blow-up criteria for local classical solutions to problem \eqref{chemo1}--\eqref{ini} that are easy to apply (see \cite{CKL13,CKL14} for the case of whole space $\mathbb{R}^d$, $d=2,3$). This is the first aim of the present paper.

Now we are in a position to state our main results.

\begin{theorem}\label{mr1} Let $\Omega\subset\mathbb{R}^3$ be a bounded domain with smooth boundary $\partial\Omega$.
Assume that $(n_0,c_0)\in C^0(\overline{\Omega})\times W^{1,q}(\Omega)$ for some $q>3$ and $n_0, c_0$ are positive.
If $T_{max}\in(0,+\infty)$ is the maximal existing time of the local classical solution $(n,c)$ to problem \eqref{chemo1}--\eqref{assmp}, then we have
\begin{equation}
	\label{cri}
     \|n(t)\|_{L^r(0,T_{max};\, L^s(\Omega))}=+\infty,\quad \mathrm{with}\;\;\frac{3}{s}+\frac{2}{r}\leq 2,\quad \frac{3}{2}<s\leq +\infty.
\end{equation}
\end{theorem}

Theorem \ref{mr1} provides a rough characterization on the blow-up rate of $\|n\|_{L^\infty(\Omega)}$ towards the blow-up time $T_{max}$ if it is finite.
Namely,
\begin{corollary}\label{re1}
Let $T_{max}\in(0,+\infty)$ be the maximal existing time of the local classical solution $(n,c)$ to problem \eqref{chemo1}--\eqref{assmp}.
For any $\delta\in (0,1)$, it holds
\begin{equation}\label{rate0}
	\limsup\limits_{t\nearrow T_{max}^-}\|n(t)\|_{L^\infty(\Omega)}(T_{max}-t)^{1-\delta}=+\infty.
\end{equation}	
\end{corollary}
Corollary \ref{re1} can be easily proved by a contradiction argument.
Indeed, if \eqref{rate0} does not hold, one then deduces that \eqref{cri} is false with $s=+\infty$, $r=1$ and hence $T_{max}$ cannot be the finite blow-up time, which leads to a contradiction of its definition. However, the estimate \eqref{rate0} does not help us to rule out the so-called type I blow-up.

To see this point, we observe that the system \eqref{chemo1} with assumption \eqref{assmp} has the following scaling invariance property (taking $\Omega=\mathbb{R}^d$): if $(n,c)$ is a solution to \eqref{chemo1}, then the pair $(n_{\lambda}, c_{\lambda})$ given by
		\begin{equation}
			n_{\lambda}(x,t)=\lambda^2 n(\lambda x,\lambda^2 t),\quad c_{\lambda}(x,t)= c(\lambda x,\lambda^2 t),\quad \forall\,\lambda>0,
		\end{equation}	
is also a solution. Motivated by this self-similar scaling invariance, a temporal blow-up rate associated with the unknown $n$ for system \eqref{chemo1} is said to be type I (or type II, respectively) if
		$\limsup\limits_{t\nearrow T_{max}^-}\|n(t)\|_{L^\infty(\Omega)}(T_{max}-t)$ is finite (or infinite, respectively).
Hence, type I blow-up solutions are those that blow up with at most a self-similar rate.
	 	
Recall the classical Keller--Segel model (see \cite{KS70})
\beq
 \begin{cases}
n_t= \Delta n-\nabla\cdot(n\nabla c),\\
\Gamma c_t= \Delta c- c+ n.
\end{cases}\nonumber
\eeq
When $\Gamma=0$, i.e., the parabolic-elliptic type, it is proved that any blow-up is type II if $d=2$ (see \cite{ST}) whereas, for $d\geq11$, radial type II blow-up solutions are known to exist (see \cite{MS}). For $3\leq d\leq 9$, a sufficient condition on the initial data ensuring type I blow-up was found in \cite{Giga}. In contrast, the situation for fully parabolic Keller--Segel system (i.e., $\Gamma>0$) is yet far from being well understood.  A lower blow-up estimate  was obtained in \cite{MizoSoup}, which provided information on $L^\infty$-norms of $(n,\nabla c)$ towards the blow-up time. However, type I blow-up cannot be ruled out. More recently, it was pointed out that only type II blow-up is possible for the classical parabolic-parabolic Keller--Segel system with nonlinear diffusion $\Delta n^m$ replacing the linear one $\Delta n$ when $1<m\leq m_c$ and $d\geq3$ where $m_c:=\frac{2(d-1)}{d}$ (see\cite{ILM16}).

Now our second main result asserts that when $d=3$, if the local classical solution to problem \eqref{chemo1}--\eqref{assmp} blows up in finite time, then the following global-in-space lower blow-up estimate can be obtained
\begin{theorem}	\label{mr2}
Suppose that the assumptions of Theorem \ref{mr1} are satisfied and the maximal existing time $T_{max}$ is finite.
Then there exists a positive constant $\alpha$ depending on $\|c_0\|_{L^\infty(\Omega)}$, $\chi$ and $\Omega$, such that
\begin{equation}
	\limsup\limits_{t\nearrow T_{max}^-} (T_{max}-t) \|n(t)\|_{L^\infty(\Omega)}  \geq \alpha.\label{lowblow}
\end{equation}
\end{theorem}
The above result indicates that if the local classical solution $(n,c)$ blows up in finite time at a self-similar rate, then there exists a $\tau\in [0,T_{max})$ such that $ \|n(t)\|_{L^\infty(\Omega)}  \geq \frac{\alpha}{2}(T_{max}-t)^{-1}$, for $t\in(\tau,T_{max})$. Besides, we will find later that for a blow-up solution, the constant $\alpha$ is proportional to $\|c_0\|^{-\frac43}_{L^\infty(\Omega)}$.
We note that the chemotaxis system \eqref{chemo1} can be viewed as a three-dimensional variant of the classical Keller--Segel model with a signal-independent sensitivity $\chi$ as well as a nonlinear chemical reaction between bacteria and the chemoattractant. This difference on the reaction term leads to different blow-up properties of the corresponding PDE system. For instance, a similar global lower blow-up rate was obtained in \cite[Theorem 1.2]{MizoSoup} for the classical Keller--Segel model, where the lower bound is independent of the initial data. However, their result fails to estimate the size of $\|n\|_{L^\infty(\Omega)}$ alone, but in terms of $L^\infty$-norm of the couple $(n, \nabla c)$ together.

Inspired by \cite{MizoSoup} for the classical Keller--Segel model, we can further prove a local version of the above non-degeneracy property Theorem \ref{mr2}.
First, we recall that for a local classical solution $(n, c)$ to problem \eqref{chemo1}--\eqref{assmp} with finite maximal time $T_{max}\in (0,+\infty)$, $x^*$ is said to be a \emph{blow-up point} if it belongs to the set
\begin{equation}
	\mathcal{B}=\bigg{\{}x^*\in\overline{\Omega}:\quad \limsup\limits_{t\nearrow T_{max}^-,\ \Omega\ni x\rightarrow x^*} \big(n(x,t)+|\nabla c(x,t)|_{\mathbb{R}^3}\big)=+\infty\bigg{\}}.\nonumber
\end{equation}
For any $x^*\in\overline{\Omega}$ and $\rho>0$, we define
\begin{equation}
	\Omega_{x^*,\rho}=B_{\rho}(x^*)\cap \overline{\Omega}
\end{equation}
where $B_{\rho}(x^*)\subset \mathbb{R}^3$ is the ball that centered at $x^*$ with radius $\rho$.
Then we have the following local non-degeneracy property for blow-up points of the solution to problem \eqref{chemo1}--\eqref{assmp}:
\begin{theorem}
	\label{mr3}
Suppose that the assumption of Theorem \ref{mr1} are satisfied and the maximal existing time $T_{max}$ is finite.
Let $x^*\in\overline{\Omega}$, $\rho>0$ and $t_0\in(0,T_{max})$. There exists a constant $\varepsilon>0$ such that, if
	\begin{equation}
		n(x,t)\leq \varepsilon(T_{max}-t)^{-1}\quad\mathrm{for\ all}\quad (x,t)\in \Omega_{x^*,\rho}\times(t_0,T_{max}),\nonumber
	\end{equation}
then $x^*$ is not a blow-up point. 	As a consequence, near any blow-up point $x^*\in \mathcal{B}$, we have the following local lower estimate:
	\begin{equation}
		\limsup\limits_{t\nearrow T_{max}^-,\ \Omega\ni x\rightarrow x^*} (T_{max}-t) n(x,t)\geq\varepsilon.\nonumber
	\end{equation}
\end{theorem}

Before concluding this section, we would like to stress some new features of the present paper.

First, Theorem \ref{mr1} provides some blow-up criteria on the bacteria density  $n$ in  critical scaling invariant spaces for arbitrary bounded smooth domain $\Omega\subset \mathbb{R}^3$ that is not necessary to be convex. Similar blow-up criteria for the Cauchy problem of system \eqref{chemofluid} with fluid interaction in $\mathbb{R}^d$ ($d=2,3$) have been obtained in \cite{CKL13,CKL14}, whose corresponding fluid-free versions are as follows:
\begin{align}
&\|\nabla c(t)\|_{L^2(0,T_{max};\, L^\infty(\mathbb{R}^d))}=+\infty,\label{criCKL}\\
&\|n(t)\|_{L^r(0,T_{max};\, L^s(\mathbb{R}^d))}=+\infty,\quad \frac{d}{s}+\frac{2}{r}\leq 2,\quad \frac{d}{2}<s\leq +\infty.\label{criCKLa}
\end{align}
We note that \eqref{criCKL} is a blow-up criterion for the chemoattractant concentration $c$ in the critical scaling-invariant space $L^2(0,T; L^\infty(\mathbb{R}^d))$.
One can easily derive a blow-up criterion corresponding to \eqref{criCKL} when the domain $\Omega\subset \mathbb{R}^3$ is smooth and bounded (see Proposition \ref{blc} below).
However, in order to derive the blow-up criterion \eqref{cri} for $n$ corresponding to \eqref{criCKLa} in bounded smooth domain, one major difficulty comes from a certain integration term on the boundary $\partial \Omega$ that takes the following form:
\beq
\int_{\partial\Omega}\frac{1}{n}\frac{\partial}{\partial\nu}|\nabla n|^2 dS.\nonumber
\eeq
 The above difficulty can be overcome by making use of the elementary lemma due to Mizoguchi \& Souplet \cite[Lemma 4.2]{MizoSoup} together with the trace theorem.
 In \cite{JWZ}, a similar difficulty concerning a boundary term for $c$ has been encountered in order to derive the energy inequality \eqref{ener0} (see Lemma \ref{lowe} below).
 Such kind of difficulties related to the boundary integration terms were avoided by imposing the essential assumption that $\Omega$ is convex in \cite{WinklerCPDE, WinkARMA}.

Second, we use a different approach to derive suitable estimates for the local classical solutions to problem \eqref{chemo1}--\eqref{assmp}.
More precisely, we choose to establish higher-order estimates of $n$ in Sobolev spaces rather than integrability in Lebesgue spaces $L^p$ for some large $p$.
This is motivated by the observation that the first equation of \eqref{chemo1} can be realized as the law of conservation of mass.
 In fact, if we introduce a new variable $\mathbf{w}$ as an ``effective velocity" such that
\begin{equation}
	\mathbf{w}= \chi\nabla c-\nabla \log n,\nonumber
\end{equation}
then the first equation of \eqref{chemo1} can be re-written into the following form
\begin{equation}
	n_t+\nabla\cdot(n\mathbf{w})=0.\nonumber
\end{equation}
Thus, one can easily deduce from system \eqref{chemo1}  a ``momentum" equation for $\mathbf{w}$.
We note that an upper bound for the kinetic energy associated with the new system of unknowns $(n,\mathbf{w})$
\begin{equation}\label{KE}
	E(n,\mathbf{w})=\frac12 \int_\Omega  n|\mathbf{w}|^2dx
\end{equation}
on $(0,T_{max})$ will imply an estimate for $\int_\Omega |\nabla \sqrt{{n}}|^2 dx$, which is one of the principal parts of $E(n,\mathbf{w})$. Due to the Sobolev embedding theorem ($d=3$), $\|n\|_{L^3(\Omega)}$ is uniformly bounded in $(0,T_{max})$  and global existence of classical solutions is a direct consequence of Corollary \ref{cor42}. Therefore, to seek a proper estimate of \eqref{KE} turns out to be crucial.
However, the effective velocity $\mathbf{w}$ and the kinetic energy $E(n,\mathbf{w})$ will not be explicitly introduced in the subsequent proofs due to possible difficulties from the boundary conditions (if $\Omega=\mathbb{T}^3$, i.e., the torus, the calculations are more straightforward). Instead, we try to control $\int_\Omega n|\nabla c|^2 dx$ and $\int_\Omega n|\nabla\log n|^2 dx$, respectively, which are the two principal parts of \eqref{KE}. The idea to establish higher-order estimates gives us a new insight on the mechanism of chemotaxis models. For example, it can provide an alternative intuitive proof for the existence of classical solutions to chemotaxis systems with the logistic source studied in \cite{LX15, WinklerCPDE10, Lank15}.

Moreover, the kinetic reformulation of the Keller--Segel model and higher-order estimates may help us to study the blow-up solutions to chemotaxis systems as well. Our approach provides a way to obtain separate lower blow-up estimates for $n$ alone.
As it has been mentioned before, lower blow-up estimates for the fully parabolic classical Keller--Segal model were recently established in \cite[Theorem 1.2]{MizoSoup} only for the couple $(n,|\nabla c|)$.
It is possible to prove lower global blow-up estimate for $n$ alone by using our method. This improvement will be illustrated in a forthcoming work.

The rest of this paper is organized as follows. In Section 2, we state some preliminary results on the local well-posedness of problem \eqref{chemo1}--\eqref{assmp}. In Section 3, we derive several higher-order estimates for the local classical solution $(n,c)$.
The last section will be devoted to the proof of our main results Theorems \ref{mr1}, \ref{mr2} and \ref{mr3}.

\section{Preliminaries}
\subsection{Notations}

Throughout this paper, we denote by $L^q(\Omega)$, $W^{k,q}(\Omega)$, $1\leq q\leq\infty$, $k\in\mathbb{N}$  the usual Lebesgue and Sobolev spaces, respectively, and as usual, $H^k(\Omega)=W^{k,2}(\Omega)$.  $\|\cdot\|_{B}$ denotes the norm in the Banach space $B$. For arbitrary vectors $\mathbf{u}=(u_1,...,u_d)^{T}, \mathbf{v}=(v_1,...,v_d)^{T}\in \mathbb{R}^d$, we denote $\mathbf{u}\cdot \mathbf{v}=\sum_{i=1}^d u_i v_i$ the inner product in $\mathbb{R}^d$, while for two $d\times d$ matrices $M_1, M_2$, we denote $M_1 : M_2=\mathrm{trace}(M_1 M_2^T)$.
For any matrix $M\in \mathbb{R}^{d\times d}$, we use the Frobenius norm $|M| =\sqrt{\mathrm{trace}(MM^T)}=\sqrt{\sum_{i,j=1}^dM_{ij}M_{ij}}$. The upper case letters $C$, $C_i$ stand for genetic constants possibly depending on the domain $\Omega$, the coefficient $\chi$ as well as the initial data. Special dependence will be pointed out explicitly in the text, if necessary.

\subsection{Local well-posedness}
It is easy to see that system \eqref{chemo1} under assumption \eqref{assmp} can be reformulated as a triangular system (see, e.g. \cite{Tao}). Then the local well-posedness of problem \eqref{chemo1}--\eqref{assmp} easily follows from the well-known parabolic regularity theory \cite{Amann} and a classical fixed point argument.
Thus, we have the following result (see e.g., \cite[Lemma 2.1]{WinklerCPDE} where a more general system with fluid interactions has been investigated, see also \cite[Lemma 2.1]{Tao}).

\begin{proposition} \label{lc}
Suppose that $\Omega\in\mathbb{R}^3$ is a bounded domain with smooth boundary.
Let $n_0$ and $c_0$ be positive and satisfy $(n_0,v_0)\in C^0(\overline{\Omega})\times W^{1,q}(\Omega)$ for some $q>3$.
Then there exists a $T_{max}>0$ and a unique local classical solution $(n,c)$ to problem \eqref{chemo1}--\eqref{assmp} such that
	\begin{equation}
		\nonumber (n,c)\in \big(C([0,T_{max}); C^0(\overline{\Omega})\times W^{1,q}(\Omega))\big) \cap \big(C^{2,1}(\overline{\Omega}\times (0, T_{max}))\big)^2.
	\end{equation}
	Moreover, $n$ and $c$ satisfy the inequalities
$$ n(x,t)> 0, \quad 0< c(x,t) \leq \|c_0\|_{L^\infty(\Omega)},\quad \mathrm{in}\ \Omega\times (0,T_{max})$$
 as well as the mass conservation property
$$\|n(t)\|_{L^1(\Omega)}=\|n_0\|_{L^1(\Omega)},\quad \forall\, t\in (0,T_{max}).$$
If $T_{max}<+\infty$, then
 \beq \label{cri00}
 \|n(t)\|_{L^\infty(\Omega)}+\|c(t)\|_{W^{1,q}(\Omega)}\nearrow +\infty, \quad\text{as}\;\;t\nearrow T_{max}^-.
 \eeq
\end{proposition}

\begin{remark}\label{REM1}
Proposition \ref{lc} implies that for any fixed $\tau_0\in (0, T_{max})$, it holds
\begin{equation}
	\sup_{0\leq t\leq \tau_0}(\|n(t)\|_{L^\infty(\Omega)}+\|c(t)\|_{W^{1,q}(\Omega)})\leq C_{\tau_0}
    \quad\text{and}\;\;\;\;(n(\tau_0), c(\tau_0))\in (C^2(\overline{\Omega}))^2,\non
\end{equation}	
 where the constant $C_{\tau_0}$ and $\|n(\tau_0)\|_{C^2(\overline{\Omega})}$, $\|c(\tau_0)\|_{C^2(\overline{\Omega})}$
 depend on $\|n_0\|_{L^\infty(\Omega)}$, $\|c_0\|_{W^{1,q}(\Omega)}$, $\Omega$, $\chi$ as well as $\tau_0$.
\end{remark}

\section{A priori estimates}

First, we recall the following lower-order energy inequality for the local classical solution $(n,c)$ to problem \eqref{chemo1}--\eqref{assmp} on any bounded smooth domain $\Omega\subset \mathbb{R}^3$, which is a special case of \cite[Lemma 3.1]{JWZ}:

\begin{lemma}\label{lowe}
Suppose that the assumptions of Theorem \ref{mr1} hold. The local classical solution $(n,c)$ to problem \eqref{chemo1}--\eqref{assmp} satisfies
\begin{align}
& \frac{d}{dt}\left\{\int_{\Omega}n\log n dx + 2 \int_\Omega |\na \sqrt{c}|^2dx \right\}\nonumber\\
& \qquad +\int_{\Omega}\frac{|\nabla n|^2}{n}dx + \frac{\chi}{2}\int_\Omega c|\nabla^2\log c|^2dx +\frac{\chi}2 \int_\Omega n\frac{|\na c|^2}{c}dx\nonumber\\
&\ \ \leq C\|\sqrt{c}\|_{L^2(\Omega)}^2 \qquad \forall\, t\in(0,T_{max}),
\label{ener0}
\end{align}
where $C>0$ is a constant depending on $\Omega$ and $\chi$, but it is independent of $t$ and $T_{max}$.
\end{lemma}

The above energy inequality together with the uniform estimate on $\|c\|_{L^\infty}$ (see Proposition \ref{lc}) 
implies the following estimates:

\begin{proposition} For any $T\in (0,T_{max})$, the local classical solution $(n,c)$ to problem \eqref{chemo1}--\eqref{assmp} satisfies
\begin{align}
& \sup\limits_{0\leq t\leq T}\left(\int_\Omega n\log n dx+ \|\nabla c\|_{L^2(\Omega)}^2\right)\leq C(1+T),\nonumber\\
& \int_0^T \int_{\Omega}\frac{|\nabla n|^2}{n}dx dt+ \int_0^T\|\Delta c\|_{L^2(\Omega)}^2dt +\int_0^T\int_\Omega n|\nabla c|^2dxdt\leq C(1+T),\nonumber
\end{align}
where $C$ is a positive constant depending on $\Omega$, $\chi$ and the initial data, but is independent of the time $T$.
\end{proposition}

Next, we derive some higher-order {\it a priori} estimates for the local classical solution $(n,c)$.
As mentioned in the introduction, we aim to control $\int_\Omega n|\nabla\log n|^2 dx$  and $\int_\Omega n|\nabla c|^2 dx$ that are the two principle parts of the kinetic energy \eqref{KE}.

To this end, the following two lemmas turn out to be useful:
\begin{lemma}(Winkler \cite[Lemma 3.3]{WinklerCPDE})
    \label{lm1}
	Let $h\in C^1(0,\infty)$ be positive and let $\Theta(s):=\int_1^s\frac{d\sigma}{h(\sigma)}$ for $s>0$. Assume that $\Omega$ is a smooth bounded domain in $\mathbb{R}^d$ with $d\geq1$. Then for any positive function $\varphi\in C^2(\overline{\Omega})$ fulfilling $\frac{\partial\varphi}{\partial\nu}=0$ on $\partial\Omega$, it holds
	\begin{equation}
		\int_\Omega \frac{h'(\varphi)}{h^3(\varphi)}|\nabla\varphi|^4 dx
       \leq (2+\sqrt{d})^2\int_\Omega \frac{h(\varphi)}{h'(\varphi)}|\nabla^2\Theta(\varphi)|^2 dx.\nonumber
	\end{equation}
\end{lemma}

\begin{lemma}(Mizoguchi \& Souplet \cite[Lemma 4.2]{MizoSoup})
 \label{MS}
 For the bounded domain $\Omega$ and $w\in C^2(\overline{\Omega})$ satisfying $\frac{\partial w}{\partial \nu}=0$ on $\partial\Omega$, we have
\beq
 \frac{\partial|\nabla w|^2}{\partial\nu}\leq2\kappa|\nabla w|^2\quad\text{on   }\partial\Omega,\nonumber
\eeq
where $\kappa=\kappa(\Omega)>0$ is an upper bound for the curvatures of $\partial\Omega$.
\end{lemma}

\textbf{First estimate.} We derive the following estimate on $\int_\Omega n|\nabla\log n|^2 dx$:

\begin{lemma}
 For any $t\in (0,T_{max})$, the local classical solution $(n,c)$ to problem \eqref{chemo1}--\eqref{assmp} satisfies the following inequality
\begin{align}
&\frac{d}{dt}\int_\Omega\frac n2|\na\log n|^2dx + \frac12 \int_\Omega n|\nabla^2\log n|^2dx\nonumber\\
&\quad \leq -\chi \int_\Omega(\nabla n\otimes\na\log n) :\nabla^2 cdx +\chi \int_\Omega\Delta n\Delta cdx+C_0\|n\|_{L^1(\Omega)},
 \label{ener1}
\end{align}
where $C_0$ is a positive constant depending only on $\Omega$.
\end{lemma}
\begin{proof}
Multiplying the first equation  in \eqref{chemo1} by $n^{-1}$, then taking gradient with respect to $x$ of the resultant, we obtain that
\beq
\label{ch1}
\pa_t\na\log n+ \chi\na c\cdot (\na^2\log n) +\chi \na\log n \cdot (\na^2 c)+ \chi \na\Delta c- \na\big(\frac{\Delta n}{n}\big)=0.
\eeq
Multiplying \eqref{ch1} by $\na n=n\na\log n$ and integrating over $\Omega$, we get
\begin{align}
&\frac{d}{dt}\int_\Omega\frac n2 |\na\log n|^2dx -\int_\Omega\frac12 |\na\log n|^2[\Delta n-\chi \nabla \cdot(n \na c)] dx
    +\chi \int_\Omega\na n\cdot \na\Delta cdx\non\\
&\quad\ \ \ +\chi \int_\Omega({\nabla n\cdot\nabla^2 c}+n\nabla c\cdot \na^2\log n)\cdot\na\log ndx
- \int_\Omega n\na\big(\frac{\Delta n}{n}\big)\cdot \na\log n dx\non\\
&\quad =0.\label{dif}
\end{align}
Using integration by parts and the boundary condition \eqref{chemo0}, we have
\beq
\label{n1}
-\int_\Omega n\na(\frac{\Delta n}{n})\cdot \na\log n dx=\int_\Omega \frac{|\Delta n|^2}{n}dx.
\eeq
Besides, a direct calculation yields that
 \beq\label{n2}
 2\frac{\Delta \sqrt{n}}{\sqrt{n}}=\frac{\Delta n}{n}-\frac{|\na n|^2}{2n^2},\quad 2n\na(\frac{\Delta\sqrt{n}}{\sqrt{n}})=\nabla \cdot (n \nabla^2\log n).
 \eeq
Then it follows from \eqref{n1}, \eqref{n2} and integration by parts that
\begin{align}
&\int_\Omega \left[-n\na(\frac{\Delta n}{n})\cdot \na\log n -\frac12 |\nabla \log n|^2 \Delta n\right] dx\non\\
&\quad =
\int_\Omega \frac{|\Delta n|^2}{n}dx-\int_\Omega\frac{|\na n|^2}{2n^2}\Delta ndx\non\\
&\quad =\int_\Omega2 \frac{\Delta \sqrt{n}}{\sqrt{n}} \Delta n dx\non\\
&\quad =-\int_\Omega2\na n\cdot \na(\frac{\Delta \sqrt{n}}{\sqrt{n}})dx\non\\
&\quad =-\int_\Omega (\na \log n)\cdot [\nabla\cdot (n \na^2\log n)]dx\non\\
&\quad =- \int_{\partial\Omega} [(n \na^2\log n)\cdot\nu]\cdot (\na \log n) ds +\int_\Omega n|\na^2\log n|^2dx\non\\
&\quad =-\frac12\int_{\partial\Omega}\frac1n\frac{\partial}{\partial\nu}|\nabla n|^2dS+\int_\Omega n|\na^2\log n|^2dx.\non
\end{align}
On the other hand, using integration by parts again, we find that
\begin{align}
&\int_\Omega \frac{\chi}{2} |\na \log n|^2\nabla \cdot (n \na c)dx + \chi \int_\Omega (n \na c\cdot \na^2\log n)\cdot \na\log ndx\non\\
&\quad =-\frac{\chi}{2}\int_\Omega n\na c\cdot\na|\na \log n|^2dx +\chi\int_\Omega (n\na c\cdot \na^2\log n)\cdot \na\log ndx\non\\
&\quad =0,\non
\end{align}
and
\begin{align}
\chi \int_\Omega\na n\cdot \na\Delta cdx=-\chi\int_\Omega\Delta n\Delta cdx.\non
\end{align}
Collecting all the above estimates, we infer from \eqref{dif} that
\begin{align}
& \frac{d}{dt}\int_\Omega\frac n2|\na\log n|^2dx+\int_\Omega n |\nabla^2\log n|^2dx\non\\
&\quad =\frac12\int_{\partial\Omega}\frac1n\frac{\partial}{\partial\nu}|\nabla n|^2dS
-\chi\int_\Omega(\nabla n\otimes\na\log n):\nabla^2 c dx
+\chi\int_\Omega\Delta n\Delta cdx.\non
\end{align}

In order to deal with the integration term on the boundary $\partial\Omega$, we make use of Lemma \ref{MS} together with the trace theorem, in the same way as in \cite[Lemma 2.4]{JWZ}, we can derive the following boundary estimate such that for any $\e\in(0,1)$
	\begin{equation}\label{nb0}
 \Bigg|\frac{1}{2}\int_{\partial\Omega}\frac1n\frac{\partial}{\partial\nu}|\nabla n|^2dS\Bigg|
     \leq \e\int_{\Omega} n|\Delta \log n|^2dx+\e\int_\Omega\frac{|\nabla n|^4}{n^3}dx+C_{\e}\|\sqrt{n}\|_{L^2(\Omega)}^2,\non
	\end{equation}
where $C_{\e}>$ is a positive constant depending only on $\Omega$ and $\e$.
 Moreover, applying Lemma \ref{lm1} and simply taking $h(s)=s$, we have
\begin{equation}
    \label{n3}
	\int_\Omega\frac{|\nabla n|^4}{n^3}dx\leq (2+\sqrt{3})^2\int_\Omega n|\nabla^2\log n|^2 dx.
\end{equation}
Due to the point-wise inequality
  $|\Delta z|^2\leq 3|\nabla^2 z|^2$ for any $z\in C^2(\overline{\Omega})$, it holds
\begin{equation}\label{ptw}
	|\Delta\log n|^2\leq 3|\nabla^2\log n|^2.\end{equation}
Thus, taking $\e>0$ small enough such that
\begin{equation}
	\left[(2+\sqrt{3})^2+3\right]\e<\frac12,\non
\end{equation}
we arrive at our conclusion \eqref{ener1}. This completes the proof.
\end{proof}

\textbf{Second estimate.} The estimate for $\int_\Omega n|\nabla c|^2 dx$ is more involved. We have

\begin{lemma}
For any $t\in(0,T_{max})$, the local classical solution $(n,c)$ of problem \eqref{chemo1}--\eqref{assmp} satisfies
\begin{align}
&\frac{d}{dt}\left(\int_\Omega \frac12 n|\nabla c|^2 +\frac {n^2c}{2\chi}+\frac{1}{\chi^2} n^2\right)dx
  +\frac{1}{2\chi}\int_\Omega cn^3dx+\frac{1}{\chi}\int_\Omega c|\nabla n|^2dx\non\\
&\qquad +\frac12\int_\Omega n|\Delta c|^2dx + \frac{2}{\chi^2}\int_\Omega |\nabla n|^2dx \non\\
&\quad \leq \e_0\int_\Omega \frac{|\Delta n|^2}{n}dx
            +\e_0\int_\Omega\frac{|\nabla n|^4}{n^3}dx
            +C_1\int_\Omega n|\na c|^4dx.\label{ccc0}
\end{align}
where $\epsilon_0$ is an arbitrary positive constant and the constant $C_1>0$ depends on $\Omega$, $\e_0$ and $\chi$.
\end{lemma}
\begin{proof}
	Multiplying the second equation in \eqref{chemo1} by $-n\Delta c$ and integrating over $\Omega$, we get
\begin{equation}
	-\int_\Omega nc_t\Delta c dx+\int_\Omega n|\Delta c|^2dx- \int_\Omega n^2c\Delta cdx=0.\non
\end{equation}
It follows from integration by parts and the boundary condition \eqref{chemo0} that
\begin{align}
	&-\int_\Omega nc_t\Delta cdx\non\\
    &\quad =\int_\Omega n(\nabla c\cdot\nabla c_t) dx+\int_\Omega (\nabla c\cdot\nabla n) c_t dx\nonumber\\
	&\quad = \frac12 \frac{d}{dt}\int_\Omega n|\nabla c|^2 dx
             -\frac12\int_\Omega |\nabla c|^2[\Delta n-\chi \nabla \cdot(n\nabla c)] dx\non\\
    &\qquad  +\int_\Omega (\nabla c\cdot\nabla n)\Delta cdx  -\int_\Omega nc(\nabla c\cdot\nabla n)dx,\nonumber
\end{align}
and
\begin{align}
	&-\int_\Omega n^2 c\Delta cdx-\int_\Omega nc(\nabla n\cdot \nabla c)dx\nonumber\\
    &\quad = \int_\Omega \nabla (n^2 c)\cdot \nabla  cdx -\int_\Omega nc\nabla n\cdot \nabla cdx\nonumber\\
	&\quad =\int_\Omega nc\nabla n\cdot \nabla cdx+\int_\Omega n^2|\nabla c|^2dx\nonumber\\
    &\quad =\int_\Omega\nabla(nc)\cdot (n\nabla c) dx\non\\
	&\quad =-\int_\Omega nc\nabla\cdot(n\nabla c)dx\nonumber\\
	&\quad =\frac{1}{\chi} \int_\Omega nc(n_t-\Delta n)dx\nonumber\\
	&\quad =\frac{d}{dt}\int_\Omega \frac{n^2c}{2\chi}dx-\int_\Omega\frac{n^2}{2\chi}(\Delta c-nc)dx
            +\frac{1}{\chi} \int_\Omega c|\nabla n|^2dx
            +\frac{1}{\chi} \int_\Omega n\nabla n\cdot\nabla cdx\nonumber\\
	&\quad =\frac{d}{dt}\int_\Omega \frac {n^2c}{2\chi}dx
            +\frac{2}{\chi}\int_\Omega n\nabla n\cdot \nabla c dx
            +\frac{1}{2\chi}\int_\Omega cn^3dx
            +\frac{1}{\chi}\int_\Omega c|\nabla n|^2dx.\non
\end{align}
As a result, we obtain that
\begin{align}
&\frac{d}{dt}\int_\Omega \frac12 n|\nabla c|^2dx
  +\frac{d}{dt}\int_\Omega \frac {n^2c}{2\chi}dx
  +\frac{1}{2\chi}\int_\Omega cn^3dx
  +\frac{1}{\chi} \int_\Omega c|\nabla n|^2dx
  +\int_\Omega n|\Delta c|^2dx\non\\
&\quad = \frac12 \int_\Omega |\nabla c|^2(\Delta n-\chi \nabla \cdot(n\nabla c))dx
        -\int_\Omega \nabla c\cdot\nabla n \Delta cdx
        -\frac{2}{\chi}\int_\Omega n\nabla n\cdot \nabla c dx.
        \label{ccc1}
\end{align}
Next, multiplying the first equation of \eqref{chemo1} by $2n$, integrating over $\Omega$,  we get
\begin{equation}
	\frac{d}{dt}\int_\Omega |n|^2 dx+2\int_\Omega |\nabla n|^2dx = 2\chi \int_\Omega n \nabla n\cdot \nabla c dx. \label{ener2aa}
\end{equation}
Multiplying  \eqref{ener2aa} by $\chi^{-2}$ and adding it up with \eqref{ccc1}, applying Young's inequality, we obtain that
\begin{align}
	&\frac{d}{dt}\int_\Omega \left(\frac12 n|\nabla c|^2 +\frac {n^2c}{2\chi}+\frac{1}{\chi^2} n^2\right)dx
            +\frac{1}{2\chi}\int_\Omega cn^3dx
            +\frac{1}{\chi}\int_\Omega c|\nabla n|^2dx\non\\
    &\qquad\  +\int_\Omega n|\Delta c|^2dx+\frac{2}{\chi^2} \int_\Omega |\nabla n|^2dx\non\\
	&\quad = \frac12 \int_\Omega |\nabla c|^2(\Delta n-\chi\nabla\cdot (n\nabla c))dx-\int_\Omega \nabla c\cdot\nabla n \Delta cdx\nonumber\\
	&\quad = \frac12 \int_\Omega |\nabla c|^2 \Delta ndx
           -\frac{\chi}{2}\int_\Omega \left(n|\nabla c|^2\Delta c+|\nabla c|^2\nabla n\cdot\nabla c\right) dx
           -\int_\Omega \nabla c\cdot\nabla n \Delta cdx \non\\
&\quad \leq \frac{1}{2} \int_\Omega n|\Delta c|^2 dx+\e_0\int_\Omega \frac{|\Delta n|^2}{n}dx+\e_0\int_\Omega\frac{|\nabla n|^4}{n^3}dx+C(\e_0,\chi)\int_\Omega n|\na c|^4 dx,\non
\end{align}
which completes the proof.	
\end{proof}

\textbf{Third estimate.} We derive some estimate for $\int_\Omega |\Delta c|^2 dx$.
\begin{lemma}
	For any $t\in (0,T_{max})$, the local classical solution $(n,c)$  to problem \eqref{chemo1}--\eqref{assmp} satisfies
\begin{align}
		&\frac{d}{dt}\int_\Omega ( |\Delta c|^2+n|\nabla c|^2) dx+\frac12\int_\Omega |\nabla c_t|^2dx
           +\frac12\int_\Omega |\na\Delta c|^2dx +\frac14\int_\Omega n|\Delta c|^2dx\non\\
		&\quad \leq  (2+\|c_0\|^2_{L^\infty})\int_\Omega  \frac{|\Delta n|^2}{n}dx
                     +\int_\Omega\frac{|\nabla n|^4}{n^3}dx+\|c_0\|^2_{L^\infty}\int_\Omega |\nabla n|^2dx\non\\
        &\qquad  +C_2\int_\Omega n|\nabla c|^4dx,
\label{cccc}
\end{align}
where $C_2>0$ is a constant depending on $\chi$.
\end{lemma}
\begin{proof}
Multiplying the second equation of \eqref{chemo1} by $-\Delta c_t$ and integrating over $\Omega$, after integration by parts, we obtain that
\begin{align}
	&\frac12\frac{d}{dt}\int_\Omega \left( |\Delta c|^2+n|\nabla c|^2\right) dx+\int_\Omega |\nabla c_t|^2dx \non\\
	&\quad = \frac12 \int_\Omega |\nabla c|^2(\Delta n-\chi \nabla \cdot (n\na c))dx
            -\int_\Omega c\nabla n\cdot \nabla c_t dx\nonumber\\
&\quad \leq \frac12\int_\Omega |\nabla c|^2(\Delta n-\chi \nabla \cdot (n\na c))dx
            +\frac12\int_\Omega |\nabla c_t|^2dx + \frac12 \int_\Omega c^2|\nabla n|^2dx \nonumber\\
&\quad = \frac12 \int_\Omega |\nabla c|^2 \Delta ndx
        -\frac{\chi}{2}\int_\Omega \left(n|\nabla c|^2\Delta c+|\nabla c|^2\nabla n\cdot\nabla c\right) dx\non\\
&\qquad
        +\frac12\int_\Omega |\nabla c_t|^2dx +\frac12 \|c\|^2_{L^\infty(\Omega)}\int_\Omega |\nabla n|^2dx\non\\
&\quad \leq \frac18 \int_\Omega n|\Delta c|^2 dx+\int_\Omega \frac{|\Delta n|^2}{n}dx+\frac{1}{4}\int_\Omega\frac{|\nabla n|^4}{n^3}dx+C\int_\Omega n|\na c|^4 dx\non\\
&\qquad +\frac12\int_\Omega|\nabla c_t|^2dx+\frac12 \|c_0\|^2_{L^\infty(\Omega)}\int_\Omega |\nabla n|^2dx,\label{ener0c}
\end{align}
where $C>0$ depends on $\chi$.

On the other hand, by taking gradient of the equation for $c$ in \eqref{chemo1}, multiplying the resultant by $\nabla\Delta c$ and integrating over $\Omega$, we deduce that
\begin{align}
	\int_\Omega |\nabla\Delta c|^2dx
    &\quad =\int_\Omega \nabla (nc)\cdot\nabla \Delta c dx+\int_\Omega\nabla c_t\cdot\nabla\Delta cdx\nonumber\\
	&\quad =-\int_\Omega \Delta (nc) \Delta c dx+\int_\Omega\nabla c_t\cdot\nabla\Delta cdx\nonumber\\
	&\quad =-\int_\Omega c\Delta n\Delta cdx-\int_\Omega n|\Delta c|^2dx-2\int_\Omega (\nabla n\cdot\nabla c) \Delta cdx\non\\
    &\qquad +\int_\Omega\nabla c_t\cdot\nabla\Delta cdx.\nonumber
	\end{align}
Hence, it holds
\begin{align}
		&\int_\Omega |\nabla\Delta c|^2dx +\int_\Omega n|\Delta c|^2dx\non\\
		&\quad =-\int_\Omega c\Delta n\Delta cdx-2\int_\Omega \nabla n\cdot\nabla c\Delta cdx+\int_\Omega\nabla c_t\cdot\nabla\Delta cdx\nonumber\\
		&\quad \leq \frac12\left(\int_\Omega n|\Delta c|^2dx
                    +\int_\Omega |\nabla\Delta c|^2dx \right)
                    +\frac12 \int_\Omega |\nabla c_t|^2dx
                    + \|c\|^2_{L^\infty(\Omega)}\int_\Omega \frac{|\Delta n|^2}{n}dx\non\\
		&\qquad +\frac12\int_\Omega\frac{|\nabla n|^4}{n^3}dx
                +C\int_\Omega n|\nabla c|^4dx,\non
	\end{align}
which together with \eqref{ener0c} yields that
\begin{align}
		&\frac12 \frac{d}{dt}\int_\Omega ( |\Delta c|^2+n|\nabla c|^2) dx
                   +\frac14\int_\Omega |\nabla c_t|^2dx +\frac14\int_\Omega |\na\Delta c|^2dx +\frac18\int_\Omega n|\Delta c|^2dx\non\\
		&\quad \leq  (1+\frac12\|c_0\|^2_{L^\infty})\int_\Omega  \frac{|\Delta n|^2}{n}dx+\frac12 \int_\Omega\frac{|\nabla n|^4}{n^3}dx
                     +\frac12\|c_0\|^2_{L^\infty}\int_\Omega |\nabla n|^2dx \non\\
        &\qquad +C\int_\Omega n|\nabla c|^4dx. \label{ener20}
\end{align}
The proof is complete.
\end{proof}


\textbf{Fourth estimate.} Combining the above estimates, we deduce the following differential inequality:

\begin{lemma}\label{lm3.8}
Let $(n,c)$ be the local classical solution to problem \eqref{chemo1}--\eqref{assmp}.
There exist constants $\kappa_1,\kappa_2>0$ such that for any $t\in(0,T_{max})$, it holds
\begin{align}
&\frac{d}{dt}\int\left[\frac n2|\na\log n|^2+\frac{\kappa_1}{2\chi}n^2c+ \frac{\kappa_1}{\chi^2}n^2+ \big(\frac{\kappa_1}{2}+\kappa_2\big)n|\na c|^2
         +\kappa_2 |\Delta c|^2\right]dx\non\\
&\qquad +\frac18\int_\Omega n |\nabla^2\log n|^2dx + \frac{\kappa_1}{\chi^2}\int_\Omega |\na n|^2dx
         +\frac{\kappa_1}{2\chi}\int_\Omega cn^3dx+\frac{\kappa_1}{\chi}\int_\Omega c|\nabla n|^2dx\non\\
&\qquad +\frac{\k_2}{2}\int_\Omega|\na c_t|^2dx+\frac{\k_2}{2}\int_\Omega |\nabla\Delta c|^2dx
        +\frac14(\kappa_1+\k_2)\int_\Omega n|\Delta c|^2dx\non\\
&\quad \leq C_3\int_\Omega n|\nabla c|^4dx +C_0\|n\|_{L^1(\Omega)},
       \label{energy00}
\end{align}
where $C_3>0$ depends on $\Omega$ and $\chi$.
\end{lemma}
\begin{proof}
Multiplying \eqref{ccc0} by $\kappa_1$, \eqref{cccc} by $\kappa_2$, respectively, with $\kappa_1,\kappa_2>0$ to be specified later, then adding the resultants to \eqref{ener1}, we obtain that
\begin{align}
&\frac{d}{dt}\int_\Omega \left[\frac{n}{2}|\na\log n|^2
                                +\frac{\kappa_1}{2\chi}n^2c+\frac{\kappa_1}{\chi^2}n^2
                                + \big(\frac{\kappa_1}{2}+\kappa_2\big) n|\na c|^2
                                +\kappa_2 |\Delta c|^2\right]dx\non\\
&\qquad +\frac12\int_\Omega n |\nabla^2\log n|^2dx
        +\frac{2\kappa_1}{\chi^2}\int_\Omega |\na n |^2dx
        +\frac{\kappa_1}{2\chi}\int_\Omega cn^3dx
        +\frac{\kappa_1}{\chi}\int_\Omega c|\nabla n|^2dx\non\\
&\qquad +\frac{\k_2}{2}\int_\Omega |\na c_t|^2dx
        +\frac{\k_2}{2}\int_\Omega |\nabla\Delta c|^2dx
        +\big(\frac{\kappa_1}{2}+\frac{\k_2}{4}\big)\int_\Omega n|\Delta c|^2dx\non\\
&\quad \leq -\int_\Omega(\nabla n\otimes\na\log n): \nabla^2 c dx
            +\chi \int_\Omega\Delta n\Delta cdx +C_0\|n\|_{L^1(\Omega)}\non\\
&\qquad +\Big[\e_0\kappa_1+(2+\|c_0\|_{L^\infty(\Omega)}^2)\k_2\Big]\int_\Omega  \frac{|\Delta n|^2}{n}dx
        +(\e_0\kappa_1+\kappa_2)\int_\Omega\frac{|\nabla n|^4}{n^3}dx\non\\
&\qquad +(C_1\k_1+C_2\k_2)\int_\Omega n|\nabla c|^4dx
        +\k_2\|c_0\|_{L^\infty(\Omega)}^2\int_\Omega |\nabla n|^2dx.
        \label{ener3}
\end{align}
Integrating by parts and applying Young's inequality, we  deduce that
\begin{align}
	&-\int_\Omega(\nabla n\otimes\na\log n):\nabla^2 cdx\non\\
    &\quad =\int_\Omega \sum_{j=1}^3 \Delta n\nabla_j\log n\nabla_j cdx+\int_\Omega \sum_{i,j=1}^3 \nabla_i n(\nabla_j\nabla_{i}\log n)\nabla_j cdx\non\\
	&\quad \leq \e_0\int_\Omega\frac{|\Delta n|^2}{n}dx+\frac{\e_0}{2}\int_\Omega\frac{|\nabla n|^4}{n^3}dx+\frac{C}{\e_0^3}\int_\Omega n|\nabla c|^4dx\nonumber\\
	&\qquad +\frac{\e_0}2\int_\Omega\frac{|\nabla n|^4}{n^3}dx+\frac18\int_\Omega n|\nabla^2\log n|^2dx+\frac{C}{\e_0}\int_\Omega n|\nabla c|^4dx\nonumber\\
	&\quad \leq\e_0\int_\Omega\frac{|\Delta n|^2}{n}dx+ \e_0\int_\Omega\frac{|\nabla n|^4}{n^3}dx+\frac18\int_\Omega n|\nabla^2\log n|^2dx+C(\e_0)\int_\Omega n|\nabla c|^4dx,\non
\end{align}
where $\e_0>0$ takes the same value as in \eqref{ener3} and
\begin{equation}
	\chi\Bigg|\int_\Omega\Delta n\Delta cdx\Bigg|\leq\e_1\int_\Omega \frac{|\Delta n|^2}{n}+\frac{\chi^2}{4\e_1}\int_\Omega n|\Delta c|^2dx,\non
\end{equation}
for some $\e_1>0$ to be determined below.

On the other hand, we deduce from the following identity
\begin{equation}
	\Delta \log n=\frac{\Delta n}{n}-\frac{|\nabla n|^2}{n^2}\non
\end{equation}
together with the Cahchy--Schwarz inequality and \eqref{n3} that
\begin{align}
	\int_\Omega \frac{|\Delta n|^2}{n}dx &\quad =\int_\Omega n\left(\Delta\log n+\frac{|\nabla n|^2}{n}\right)^2 dx\non\\
	&\quad \leq 2\int_\Omega n|\Delta\log n|^2dx+2\int_\Omega \frac{|\nabla n|^4}{n^3}dx\non\\
	&\quad \leq 6\int_\Omega n|\nabla^2\log n|^2dx+2\int_\Omega \frac{|\nabla n|^4}{n^3}dx\non\\
    &\quad \leq \left[6+2(2+\sqrt{3})^2\right]\int_\Omega n|\nabla^2\log n|^2dx,\label{n4}
\end{align}	
where we have also used the point-wise inequality $|\Delta \log n|^2\leq 3|\nabla ^2\log n|^2$.

Recalling the estimates \eqref{n3} and \eqref{n4}, we first take the positive constants $\e_1,\kappa_2$ small enough, for instance,
\begin{align}
&\left[6+2(2+\sqrt{3})^2\right]\e_1=\frac{1}{24},\non\\
&(2+\|c_0\|_{L^\infty}^2)\left[6+2(2+\sqrt{3})^2\right]\k_2\leq \frac{1}{24},\non\\
&(2+\sqrt{3})^2\kappa_2\leq \frac{1}{24},\non
\end{align}
such that the following inequality holds
\begin{align}
	\left[\e_1+(2+\|c_0\|_{L^\infty}^2)\k_2\right]\int_\Omega \frac{|\Delta n|^2}{n}dx+\kappa_2\int_\Omega\frac{|\nabla n|^4}{n^3}dx \leq \frac18 \int_\Omega n |\nabla^2\log n|^2dx.\non
\end{align}
In particular, we note that $\k_2\|c_0\|_{L^\infty}^2$ is bounded by a small constant $C_4$ that is independent of $\|c_0\|_{L^\infty}$.
After fixing $\e_1, \kappa_2$, we choose $\kappa_1>0$ large enough such that
\begin{equation}
	\frac{\kappa_1}{4}>\frac{\chi^2}{4\e_1}\quad\text{and}\quad \frac{\kappa_1}{\chi^2}> C_4\geq \kappa_2\|c_0\|^2_{L^\infty},\non
\end{equation}
We note that $\kappa_1$ may depend on $\Omega$ and $\chi$, but is independent of $\|c_0\|_{L^\infty(\Omega)}$.
Finally, keeping \eqref{n3} and \eqref{n4} in mind, we choose $\e_0$ sufficiently small such that
\begin{equation}
	(\e_0\kappa_1+\e_0)\int_\Omega  \frac{|\Delta n|^2}{n}dx+(\e_0\kappa_1+\e_0)\int_\Omega\frac{|\nabla n|^4}{n^3}dx\leq \frac18 \int_\Omega n |\nabla^2\log n|^2dx.\non
\end{equation}
Again, $\e_0$ may depend on $\Omega$ and $\chi$, but it is independent of $\|c_0\|_{L^\infty(\Omega)}$.
Combing the above estimates together, we conclude \eqref{energy00}.

The proof is complete.
\end{proof}


\textbf{Fifth estimate.} In order to treat the reminder term $\int_\Omega n|\nabla c|^4dx$ on the right-hand side of \eqref{energy00} and close the estimate,
we shall derive the following differential inequality for $\|\nabla c\|_{L^4(\Omega)}^4$:

\begin{lemma}\label{lm3.9}
For any $t\in[0,T_{max})$, the following inequality holds
\begin{align}
& \frac{d}{dt}\int_\Omega |\nabla c|^4 dx
  +\int_\Omega\big|\nabla|\nabla c|^2\big|^2dx
  +4\int_\Omega |\nabla^2c|^2|\nabla c|^2 dx
  +4\int_\Omega n |\nabla c|^4 dx\non\\
&\quad \leq \delta_0\int_\Omega\frac{|\nabla n|^4}{n^3} dx + C_5\int_\Omega |\nabla c|^4dx
        +3\left(\delta_0^{-1}\|c_0\|^4_{L^\infty(\Omega)}\right)^\frac13 \int_\Omega n|\nabla c|^4dx,
        \label{c4}
\end{align}
where $\delta_0$ is an arbitrary positive constant and $C_5>0$ depends on $\Omega$.
\end{lemma}
\begin{proof}
	Taking gradient of the second equation in \eqref{chemo1}, multiplying the resultant by $|\nabla c|^2\nabla c$ and integrating over $\Omega$, we get
\begin{align}
& \frac{1}{4}\frac{d}{dt}\int_\Omega |\nabla c|^4 dx
  +\frac12\int_\Omega\big|\nabla|\nabla c|^2\big|^2dx
  +\int_\Omega |\nabla^2c|^2|\nabla c|^2 dx
  +\int_\Omega n |\nabla c|^4 dx\non\\
&\quad =\int_{\partial\Omega}\frac12 |\nabla c|^2 \frac{\partial}{\partial\nu}|\nabla c|^2dS
         -\int_\Omega c|\nabla c|^2 \nabla n \cdot\nabla c dx,\label{nc1}
\end{align}
where we have used the identity $-\nabla \Delta c\cdot \nabla c= -\frac12 \Delta |\nabla c|^2+|\nabla^2 c|^2$ and integration by parts.

Using Lemma \ref{MS} and the trace theorem, we deduce that for some $s_0\in(0,\frac12)$,
\begin{align}
	\int_{\partial\Omega}\frac12 |\nabla c|^2 \frac{\partial}{\partial\nu}|\nabla c|^2dS
 \leq \kappa(\Omega)\int_{\partial\Omega}|\nabla c|^4ds\leq C(\Omega)\||\nabla c|^2\|^2_{H^{\frac12+s_0}(\Omega)}.\non
\end{align}
It follows from the interpolation inequality that
\begin{align}
	\||\nabla c|^2\|^2_{H^{\frac12+s_0}(\Omega)}
&\leq C\||\nabla c|^2\|^{1+2s_0}_{H^1(\Omega)}\||\nabla c|^2\|_{L^2(\Omega)}^{1-2s_0}\non\\
&\leq C\|\nabla|\nabla c|^2\|_{L^2(\Omega)}^{1+2s_0}\||\nabla c|^2\|_{L^2(\Omega)}^{1-2s_0}+C\||\nabla c|^2\|_{L^2(\Omega)}^{2}.\non
\end{align}
Therefore, by Young's inequality, we obtain from \eqref{nc1} that
\begin{align}
&\frac{1}{4}\frac{d}{dt}\int_\Omega |\nabla c|^4 dx+\frac12\int_\Omega\big|\nabla|\nabla c|^2\big|^2dx+\int_\Omega |\nabla^2c|^2|\nabla c|^2 dx
           +\int_\Omega n|\nabla c|^4dx\non\\
&\quad \leq \frac14 \int_\Omega\big|\nabla|\nabla c|^2\big|^2dx
             +C\int_\Omega |\nabla c|^4dx
             +\frac{\delta_0}{4}\int_\Omega\frac{|\nabla n|^4}{n^3} dx
             +\frac34\delta_0^{-\frac13}\int_\Omega c^{\frac43} n|\nabla c|^4dx\non\\
&\quad \leq \frac14 \int_\Omega\big|\nabla|\nabla c|^2\big|^2dx
             +\frac{\delta_0}{4}\int_\Omega\frac{|\nabla n|^4}{n^3} dx
             +C\int_\Omega |\nabla c|^4dx \non\\
&\qquad      +\frac34\left(\delta_0^{-1}\|c_0\|^4_{L^\infty(\Omega)}\right)^\frac13\int_\Omega n|\nabla c|^4dx,
             \label{keyest000}
\end{align}
which implies \eqref{c4}. The proof is complete.
\end{proof}

\textbf{Sixth estimate.} Finally, combing the above differential inequalities, we are able to conclude that

\begin{lemma}\label{LEM}
Let $(n,c)$ be a local classical solution to problem \eqref{chemo1}--\eqref{assmp} and $\kappa_3>0$ be an arbitrary positive constant. Then it holds
\begin{align}
&\frac{d}{dt}\int_\Omega
               \left[\frac n2|\na\log n|^2+\frac{\kappa_1}{2\chi}n^2c+\frac{\kappa_1}{\chi^2}n^2
         + \big(\frac{\kappa_1}{2}+\kappa_2\big)n|\na c|^2 +\kappa_2|\Delta c|^2+\kappa_3|\nabla c|^4\right]dx\non\\
&\qquad  + \frac{\kappa_1}{\chi^2}\int_\Omega |\na n|^2dx
         +\frac{\kappa_1}{2\chi}\int_\Omega cn^3dx+\frac{\kappa_1}{\chi}\int_\Omega c|\nabla n|^2dx+\frac{\k_2}{2}\int_\Omega|\na c_t|^2dx\non\\
&\qquad +\frac{\k_2}{2}\int_\Omega |\nabla\Delta c|^2dx
        +\frac14(\kappa_1+\k_2)\int_\Omega n|\Delta c|^2dx +\kappa_3\int_\Omega\big|\nabla|\nabla c|^2\big|^2dx\non\\
&\qquad +4\kappa_3\int_\Omega |\nabla^2c|^2|\nabla c|^2 dx\non\\
&\quad \leq  \left[(2+\sqrt{3})^2 \kappa_3\delta_0-\frac18\right]\int_\Omega n |\nabla^2\log n|^2dx \non\\
&\qquad +\left[C_3 +3\kappa_3\left(\delta_0^{-1}\|c_0\|^4_{L^\infty(\Omega)}\right)^\frac13-4\kappa_3\right] \int_\Omega n|\nabla c|^4dx\non\\
&\qquad + \kappa_3 C_5\int_\Omega |\nabla c|^4dx +C_0\|n\|_{L^1(\Omega)}, \label{energy001}
\end{align}
where the coefficients $\kappa_1$, $\kappa_2$, $\delta_0$, $C_0$, $C_3$, $C_5$ are chosen as in Lemmas \ref{lm3.8}, \ref{lm3.9}.
\end{lemma}
\begin{proof}
Multiplying \eqref{c4} by an arbitrary constant $\kappa_3>0$, adding the resultant to \eqref{energy00} and using the estimate \eqref{n3}, we arrive at our conclusion.
\end{proof}

\section{Proof of Main Results}

\subsection{A blow-up criterion in terms of the unknown variable $c$}

It has been shown in \cite{Tao} that the estimate of $\|n\|_{L^{d+1}(\Omega)}$ ($d\geq 2$) plays an important role in the proof of global existence of classical solutions to problem \eqref{chemo1}--\eqref{assmp}.
Adapting the argument there to our current case $d=3$, we have

\begin{lemma}\label{eL4}
Let the assumptions of Theorem \ref{mr1} hold. If
\begin{align}
\|n(\cdot, t)\|_{L^4(\Omega)}\leq C_6, \quad \forall\, t\in [0,T_{max}),\label{nL4}
\end{align}
for some  constant $C_6>0$, then there exists a constant $C_7>0$ such that
\begin{align}
\|n(\cdot, t)\|_{L^\infty(\Omega)}\leq C_7, \quad \forall\, t\in [0,T_{max}).\label{nLinf}
\end{align}
\end{lemma}
To prove Lemma \ref{eL4}, one can first apply the semigroup approach to establish a uniform bound for
$\|c\|_{W^{1,\infty}(\Omega)}$ and then employ the iterative technique of Moser \& Alikakos (see e.g., \cite{Ali}) to derive a uniform bound on $\|n\|_{L^\infty(\Omega)}$.
The proof is omitted here and we refer the reader to \cite[Lemma 3.2]{Tao} for details.

\begin{remark}
Noticing the uniform estimate $\|c(t)\|_{L^\infty}\leq \|c_0\|_{L^\infty}$, the system \eqref{chemo1}--\eqref{assmp} fulfills those conditions in \cite[Lemma 3.2]{BBTW} with the choice of parameters $\alpha=0$ and $\beta=1$. Then a blow-up criterion for classical solution to our system in terms of $n$ is given by
\begin{equation}
	\limsup_{t\nearrow T_{max}^-}\|n(t)\|_{L^{\frac{d}{2}+\e}(\Omega)}=+\infty, \quad \text{for}\ \e>0.\label{subcri0}
\end{equation}
Here we note that $L^\infty(0,T; L^{\frac{d}{2}}(\Omega))$ is a critical scaling-invariant space for $n$, thus \eqref{subcri0} is given in a subcritical space.
The criterion \eqref{subcri0} was obtained in \cite{BBTW} by semigroup method with some delicate estimates.
Below we shall see that this criterion in three-dimensional case will be recovered in Corollary \ref{cor42}, by the standard energy method and our higher-order energy estimate \eqref{energy001}.
\end{remark}

Now, we can derive a blow-up criterion for the unknown  $c$ corresponding to \eqref{criCKL} when $\Omega\subset \mathbb{R}^3$ is a bounded smooth domain.

\begin{proposition}\label{blc}
Under the same assumptions of Theorem \ref{mr1}, if $T_{max}\in(0,+\infty)$ is the maximal existing time of the local classical solution to problem \eqref{chemo1}--\eqref{assmp}, then
\begin{align}
&\|\nabla c(t)\|_{L^2(0,T_{max};\, L^\infty(\Omega))}=+\infty.\label{criCKL1}
\end{align}
\end{proposition}
\begin{proof}
Multiplying the first equation of \eqref{chemo1} by $n^{3}$, integrating over $\Omega$, after integration by parts, we get
\begin{align}
	\frac14\frac{d}{dt}\int_\Omega n^4dx+3\int_\Omega n^{2}|\nabla n|^2dx
    &\quad = 3\chi \int_\Omega n^{3}\nabla n\cdot \nabla cdx\nonumber\\
	&\quad \leq \frac{3}{2}\int_\Omega n^{2}|\nabla n|^2dx+\frac{3}{2}\chi^2\int_\Omega n^4|\nabla c|^2dx\nonumber\\
	&\quad \leq \frac{3}{2}\int_\Omega n^{2}|\nabla n|^2dx+\frac{3}{2}\chi^2 \|\nabla c\|^2_{L^\infty(\Omega)}\int_\Omega n^4dx.
   \label{energyabc}
\end{align}
If the criterion \eqref{criCKL1} is not valid, i.e., there exists a finite positive constant $C_8$ such that
\beq
\int_0^{T_{max}}\|\nabla c(t)\|^2_{L^\infty}dt\leq C_8,\label{enc}
\eeq
then we infer from \eqref{energyabc} and Gronwall's lemma that
\begin{equation}
	\int_\Omega n(t)^4dx<+\infty,\quad \forall\, t\in [0,T_{\max}).\non
\end{equation}
Once we get the estimate of $n$ in $L^\infty(0,T;L^4(\Omega))$, we can deduce from Lemma \ref{eL4} the bound of $n$ in $L^\infty(0,T_{max};L^\infty(\Omega))$, i.e., \eqref{nLinf}.
Then it follows from the local well-posedness result  that the local classical solution $(n,c)$ can be extended beyond $T_{max}$, which leads to a contradiction.

The proof is complete.
\end{proof}

\subsection{Proof of Theorem \ref{mr1}}
\noindent
Now we are in a position to prove Theorem \ref{mr1}.
Suppose that $T_{max}\in(0,+\infty)$ and \eqref{cri} is not valid, i.e.,
\beq \label{cri0}
\|n\|_{L^{r}(0,T_{max};\,L^s(\Omega))}<+\infty,\qquad\text{for}\;\;\frac{2}{s}+\frac{2}{r}\leq 2,\;\;\frac32<s\leq +\infty.
\eeq

We first derive the following estimates:
\begin{lemma}
If $T_{max}\in(0,+\infty)$, under the assumption \eqref{cri0}, for arbitrary but fixed time $\tau_0\in (0, \frac12 T_{max})$, the following estimates hold:
\begin{equation}\label{est1}
			\sup_{\tau_0\leq t\leq T_{max}}\int_\Omega\left(n|\na\log n|^2+|n|^2+ n|\na c|^2+|\Delta c|^2+|\nabla c|^4\right)dx	\leq C,
\end{equation}
	and
	\begin{align}\label{est2}
		&\int_{\tau_0}^{T_{max}}\left(\int_\Omega |\na n|^2dx +\int_\Omega n |\nabla^2\log n|^2dx +\int_\Omega |\na c_t|^2dx +\int_\Omega |\nabla\Delta c|^2dx \right.\non\\
		&\qquad\qquad  \left.+\int_\Omega n|\Delta c|^2dx+\int_\Omega c n^3 dx+\int_\Omega\big|\nabla|\nabla c|^2\big|^2dx+\int_\Omega n|\nabla c|^4dx \right)dt\leq C,
	\end{align}
where $C>0$ depends on $\|n_0\|_{L^\infty(\Omega)}$, $\|c_0\|_{W^{1,q}(\Omega)}$, $\Omega$, $\chi$,  $\tau_0$ as well as $T_{max}$.
\end{lemma}
\begin{proof}
In Lemma \ref{LEM}, first taking $\kappa_3=1$ and then setting the parameter $\delta_0$ to be sufficiently small, we can deduce from \eqref{energy001} that for some $C_9>0$,
\begin{align}
&\frac{d}{dt}V(t)+G(t) \leq C_9 \int_\Omega n|\nabla c|^4dx+ C_5\int_\Omega |\nabla c|^4dx +C_0\|n_0\|_{L^1(\Omega)}. \label{energy002}
\end{align}
where
\begin{align}
&V(t)=\int_\Omega \left[\frac n2|\na\log n|^2+\frac{\kappa_1}{2\chi}n^2c+\frac{\kappa_1}{\chi^2}n^2
         + \big(\frac{\kappa_1}{2}+\kappa_2\big)n|\na c|^2 +\kappa_2|\Delta c|^2+\kappa_3|\nabla c|^4\right]dx,\label{V}
\end{align}
and
\begin{align}
 G(t)& = \frac{\kappa_1}{\chi^2}\int_\Omega |\na n|^2dx
         +\frac{\kappa_1}{2\chi}\int_\Omega cn^3dx+\frac{\kappa_1}{\chi}\int_\Omega c|\nabla n|^2dx+\frac{\k_2}{2}\int_\Omega|\na c_t|^2dx\non\\
&\quad +\frac{\k_2}{2}\int_\Omega |\nabla\Delta c|^2dx
        +\frac14(\kappa_1+\k_2)\int_\Omega n|\Delta c|^2dx +\kappa_3\int_\Omega\big|\nabla|\nabla c|^2\big|^2dx\non\\
&\quad +4\kappa_3\int_\Omega |\nabla^2c|^2|\nabla c|^2 dx+\frac{1}{16}\int_\Omega n |\nabla^2\log n|^2dx.\label{G}
\end{align}

\noindent Below we consider two cases.

 \textit{Case 1}. Suppose \eqref{cri0} holds with $s=+\infty$, $r\geq1$.

Using the estimate
\begin{equation}
	\int_\Omega n|\na c|^4dx \leq \|n\|_{L^\infty(\Omega)}\int_\Omega |\na c|^4dx,\label{nnn}
\end{equation}
we obtain that for any $t\in(0,T_{max})$
\begin{equation}
	\frac{d}{dt}V(t)+G(t)\leq C_{10}(1+\|n\|_{L^\infty(\Omega)}) V(t)+C_0\|n_0\|_{L^1(\Omega)}.\non
\end{equation}
Remark \ref{REM1} implies that $V(\tau_0)<+\infty$. Then under the assumption \eqref{cri0}, we conclude from Gronwall's lemma and the local estimate obtained in Remark \ref{REM1} that
\beq \sup_{\tau_0\leq t\leq T_{max}}V(t)+\int_{\tau_0}^{T_{max}}G(t)dt\leq C,\non
\eeq
which gives the estimates \eqref{est1} and \eqref{est2}.

\textit{Case 2.} Suppose \eqref{cri0} holds for $s\in (\frac32, +\infty)$. By the H\"{o}lder inequality, we have
\begin{equation}
	\int_\Omega n|\nabla c|^4dx\leq\|n\|_{L^s(\Omega)}\|\nabla c\|^4_{L^{4s'}(\Omega)},\quad \mathrm{with}\ \ \frac{1}{s}+\frac{1}{s'}=1,\non
\end{equation}
and by the Sobolev embedding theorem ($d=3$), it holds
\begin{equation}
	\|\nabla c\|^{4}_{L^{4s'}(\Omega)}\leq C\|\nabla|\nabla c|^2\|_{L^2(\Omega)}^{\frac{3}{s}}\|\nabla c\|_{L^4(\Omega)}^{4-\frac{6}{s}}+C\|\nabla c\|^4_{L^4(\Omega)}.\non
\end{equation}
Then it follows from Young's inequality that for any $\epsilon>0$
\begin{align}
		\int_\Omega n|\nabla c|^4dx &\leq \|n\|_{L^s(\Omega)}\|\nabla c\|^4_{L^{4s'}(\Omega)}\nonumber\\
		&\leq \epsilon\|\nabla|\nabla c|^2\|_{L^2(\Omega)}^2+C_{\epsilon}\bigg(\|n\|_{L^s(\Omega)}^{\frac{2s}{2s-3}}+\|n\|_{L^s(\Omega)}\bigg)\|\nabla c\|^4_{L^4(\Omega)},\label{nn2n}
\end{align}
 provided that $s>\frac{3}{2}$.
Since $r\geq \frac{2s}{2s-3}>1$, we infer from the assumption \eqref{cri0} that
\begin{equation}
	\int_0^{T_{max}}\bigg(\|n\|_{L^s(\Omega)}^{\frac{2s}{2s-3}}+\|n\|_{L^s(\Omega)}\bigg)dt<+\infty.\label{nn3n}
\end{equation}
Therefore, picking $\epsilon$ sufficiently small in \eqref{nn2n}, then again by Gronwall's lemma we may deduce from \eqref{energy002}, \eqref{nn3n} that the estimates \eqref{est1} and \eqref{est2} hold.

The proof is complete.
\end{proof}

Next, the estimates \eqref{est1}, \eqref{est2} and the Sobolev embedding theorem ($d=3$) yield that
\beq
\int_{\tau_0}^{T_{\max}} \|\nabla c(t) \|^2_{L^\infty(\Omega)}dt
\leq C(\Omega)\int_{\tau_0}^{T_{\max}}(\|\nabla \Delta c(t)\|_{L^2(\Omega)}^2+\|c(t)\|_{L^2(\Omega)}^2)dt
\leq C.\label{enc1}
\eeq
 Using a similar argument for Proposition \ref{blc} and Lemma \ref{eL4}, we can show that
 $$\sup\limits_{\tau_0\leq t\leq T_{max}}\|n(t)\|_{L^\infty(\Omega)}\leq C,$$
  which combined with the estimate of $\|n(t)\|_{L^\infty(\Omega)}$ on $[0,\tau_0]$ given in Remark \ref{REM1} implies the uniform bound
\begin{equation}
	\sup\limits_{0\leq t\leq T_{max}}\|n(t)\|_{L^\infty(\Omega)}\leq C.\non
\end{equation}

As a result, one can extend the local classical solution $(n,c)$ beyond $T_{max}$,
which contradicts with the definition that $T_{max}\in (0,+\infty)$ is the maximal existing time. The proof of Theorem \ref{mr1} is then complete. $\square$
\medskip

As a by-product of the differential inequality \eqref{energy001}, we also have the following global existence result for small initial data:

\begin{corollary}\label{smallini}
Let the assumptions of Theorem \ref{mr1} hold. There exits a positive constant $\sigma_0$ depending on $\Omega$ and $\chi$ such that if $\|c_0\|_{L^\infty(\Omega)}\leq \sigma_0$, then problem \eqref{chemo1}--\eqref{assmp} admits a unique global classical solution $(n,c)$.
\end{corollary}
\begin{proof}
In the inequality \eqref{energy001}, we take the parameters $\kappa_3$ and $\delta_0$ satisfying
$$\kappa_3=C_3, \quad (2+\sqrt{3})^2\kappa_3\delta_0=\frac{1}{16},$$
where we recall that $C_3$ depends only on $\Omega$ and $\chi$.
Let $\sigma_0>0$ be a constant such that
$$3\left(\delta_0^{-1}\sigma_0^4\right)^\frac13=1.$$
Then for any initial data satisfying $\|c_0\|_{L^\infty(\Omega)}\leq \sigma_0$, we infer from \eqref{energy001} that
\begin{align}
&\frac{d}{dt}V(t)+G(t) \leq C_3C_5\int_\Omega |\nabla c|^4dx +C_0\|n_0\|_{L^1(\Omega)}, \non
\end{align}
where $V(t)$ and $G(t)$ are given in \eqref{V} and \eqref{G}, respectively (with the only difference that here $\kappa_3=C_3$).
It easily follows from Gronwall's lemma that \eqref{enc1} holds.
As a consequence, we can argue as in the proof of Theorem \ref{mr1} to show that $T_{max}=+\infty$ and $(n,c)$ is indeed a global classical solution. The proof is complete.
\end{proof}

On the other hand, we would like to show that the blow-up criterion \eqref{subcri0} can be recovered by standard energy method.
\begin{corollary}\label{cor42}
Let $\e>0$. Suppose that the assumptions of Theorem \ref{mr1} are satisfied. If $T_{max}\in(0,+\infty)$ is the maximal existing time of the local classical solution $(n,c)$ to problem \eqref{chemo1}--\eqref{assmp}, then it holds
\begin{equation}	
\limsup_{t\nearrow T_{max}^-}\|n(t)\|_{L^{\frac{3}{2}+\e}(\Omega)}=+\infty.\label{subcri}
\end{equation}	
\end{corollary}
\begin{proof}
 By the Sobolev embedding theorem $(d=3)$, it holds that
	\begin{equation}
	\|\nabla c\|_{L^{12}(\Omega)}^2=\||\nabla c|^2\|_{L^6(\Omega)}\leq C(\|\nabla |\nabla c|^2\|_{L^2(\Omega)}+\||\nabla c|^2\|_{L^2(\Omega)}).
\end{equation}
On the other hand, for any $t\in [0,T_{max})$,
\begin{align}
	\int_\Omega n|\nabla c|^4dx
    &=\int_{\{x\in\Omega,\, n(x,t)\leq N\}} n|\nabla c|^4dx
     +\int_{\{x\in\Omega,\, n(x,t)>N\}} n|\nabla c|^4dx\non\\
	&\leq N\int_{\Omega}|\nabla c|^4dx+\int_{\{x\in\Omega,\, n(x,t)>N\}} n|\nabla c|^4dx,
\end{align}
 where $N$ is a positive number to be chosen below. For convenience, we denote $\Omega_N^t=\{x\in\Omega,\, n(x,t)>N\}$.
 Then $|\Omega_N^t|\leq N^{-1}\|n_0\|_{L^1(\Omega)}$ since the total mass of $n$ is conserved.
 Moreover, it follows from the H\"{o}lder inequality that
\begin{align}
	\int_{\Omega_N^t} n|\nabla c|^4dx	
    &\leq \bigg(\int_{\Omega_N^t} n^{\frac{3}{2}}dx\bigg)^{\frac{2}{3}}\bigg(\int_{\Omega_N^t}|\nabla c|^{12}dx\bigg)^{\frac{1}{3}}\non\\
    &\leq \bigg(\int_{\Omega_N} n^{\frac{3}{2}}dx\bigg)^{\frac{2}{3}}\bigg(\int_{\Omega}|\nabla c|^{12}dx\bigg)^{\frac{1}{3}}\non\\
	&\leq \widetilde{C}\bigg(\int_{\Omega_N} n^{\frac{3}{2}}dx\bigg)^{\frac{2}{3}}\left(\|\nabla |\nabla c|^2\|_{L^2(\Omega)}^2+\|\nabla c\|_{L^4(\Omega)}^4\right).
\end{align}
If \eqref{subcri} is false, then there is  $M>0$, such that for any $T\in(0,T_{max})$,
\begin{equation}
	\sup\limits_{0\leq t\leq T}\|n(t)\|_{L^{\frac{3}{2}+\e}(\Omega)}\leq M.
\end{equation}
As a result, for any $t\in(0, T_{max})$
\begin{align}
	\bigg(\int_{\Omega_N^t} n^{\frac{3}{2}}dx\bigg)^{\frac{2}{3}}&\leq \bigg(\int_{\Omega_N^t} n^{\frac{3}{2}+\e}dx\bigg)^{\frac{2}{3+2\e}}\bigg{|}\Omega_N^t\bigg{|}^{\frac{4\e}{9+6\e}}\non\\
	&\leq MN^{-\frac{4\e}{9+6\e}}\|n_0\|^{\frac{4\e}{9+6\e}}_{L^1(\Omega)}.
\end{align}
Now in the differential inequality \eqref{energy001}, we first choose $\kappa_3$, $\delta_0 >0$ satisfying
\begin{align}
	&3\left(\delta_0^{-1}\|c_0\|^4_{L^\infty(\Omega)}\right)^\frac13=4,\quad (2+\sqrt{3})^2\kappa_3\delta_0=\frac{1}{16},\non
\end{align}
and then take $N>0$ sufficiently large such that
\begin{equation}
	C_3\widetilde{C}MN^{-\frac{4\e}{9+6\e}}\|n_0\|^{\frac{4\e}{9+6\e}}_{L^1(\Omega)}\leq\frac{\kappa_3}{2},
\end{equation}
 we thus infer from \eqref{energy001} that
\begin{align}
&\frac{d}{dt}V(t)+\frac12G(t) \leq C(N)\int_\Omega |\nabla c|^4dx +C_0\|n_0\|_{L^1(\Omega)}.
\end{align}
Again, we can argue as the proof of Theorem \ref{mr1} to show that the local classical solution  $(n,c)$ can be extended beyond  $T_{max}$, which leads to a contradiction.

The proof is complete.
\end{proof}

\subsection{Proof of Theorem \ref{mr2}}
We prove Theorem \ref{mr2} by a contradiction argument.
Recalling the differential inequality \eqref{energy001}, we choose $\kappa_3$ and $\delta_0$ such that
\begin{equation}		
\kappa_3=\frac{C_3}{4},\quad (2+\sqrt{3})^2\kappa_3\delta_0=\frac{1}{16},\label{kkk3}
\end{equation}
where $C_3>0$ is a constant depending on $\Omega$ and $\chi$, which is given in Lemma \ref{lm3.8}.
Then letting
\begin{equation}
	\widetilde{V}(t)=V(t)+\|n_0\|_{L^1(\Omega)}>0,
\end{equation}
we infer from \eqref{energy001} that	
\begin{equation}
	\label{bu00}
     \frac{d}{dt}\widetilde{V}(t)+G(t)
        \leq \widetilde{C}\kappa_3\int_{\Omega}n|\nabla c|^4dx+C_5\kappa_3\int_\Omega |\nabla c|^4dx+C_0\|n_0\|_{L^1(\Omega)},
\end{equation} where
\begin{equation}
\widetilde{C}=3\delta_0^{-\frac{1}{3}}\|c_0\|_{L^\infty(\Omega)}^{\frac{4}{3}}=3\big[2(2+\sqrt{3})\big]^{\frac{2}{3}}C_3^{\frac{1}{3}}\|c_0\|_{L^\infty(\Omega)}^{\frac{4}{3}}.
\end{equation}
Denote
\begin{align}
\alpha:=\frac{1}{4\widetilde{C}}.\label{bula}
\end{align}
Below we will prove that if $T_{max}\in (0,+\infty)$ is the maximal existing time, then
\begin{equation}\label{bulb}
	\limsup\limits_{t\nearrow T_{max}^-} (T_{max}-t) \|n(t)\|_{L^\infty(\Omega)} \geq\alpha.
\end{equation}
If \eqref{bulb} is not true, there must exist a constant $\alpha_1\in(0,\alpha)$ such that
$$\limsup\limits_{t\nearrow T_{max}^-} (T_{max}-t) \|n(t)\|_{L^\infty(\Omega)} =\alpha_1.$$
Thus for $\alpha_0=\frac{\alpha+\alpha_1}{2}\in(\alpha_1,\alpha)$,  there exists a time $\tau_1\in [\frac12T_{max},T_{max})$ that may depend on $\alpha_0$  such that
\begin{equation}
\|n(t)\|_{L^\infty(\Omega)}\leq \alpha_0(T_{max}-t)^{-1}, \quad \forall\, t\in [\tau_1,T_{max}).\label{bu1}
\end{equation}
It follows  that
\begin{equation}
	\label{bu2}
    \int_\Omega n|\nabla c|^4dx\leq \alpha_0(T_{max}-t)^{-1}\int_{\Omega}|\nabla c|^4dx,\quad \forall\, t\in[\tau_1,T_{max}),
\end{equation}
and hence we infer from \eqref{bu00} that for $ t\in[\tau_1,T_{max})$,
\begin{align}
\frac{d}{dt}\tilde{V}(t)+G(t)
        &\leq \widetilde{C}\alpha_0(T_{max}-t)^{-1}\left(\kappa_3\int_{\Omega}|\nabla c|^4dx\right)
         +C_5\left(\kappa_3\int_\Omega |\nabla c|^4dx\right)\non\\
        &\quad +C_0\|n_0\|_{L^1(\Omega)}.\label{bu0}
\end{align}
Let $\eta>0$ be the constant such that
\begin{equation}
	(1+2\eta)\alpha_0=\alpha.
\end{equation}
We can find $\tau_2\geq \tau_1$ that satisfies
\begin{equation}
	\max\{C_5, C_0\}\leq \eta \widetilde{C} \alpha_0(T_{max}-t)^{-1},\quad\forall\,t\in[\tau_2, T_{max}).\non
\end{equation}
Then it follows from \eqref{bu0} and the definition of $\tilde{V}(t)$  that
\begin{equation}
	\label{bu3}
	\frac{d}{dt}\log\tilde{V}(t)\leq (1+\eta)\widetilde{C}\alpha_0(T_{max}-t)^{-1},\quad\forall\,t\in[\tau_2, T_{max}).
\end{equation}
Noticing the fact
\begin{equation}
	\int_{\tau_2}^t (T_{max}-s)^{-1}ds=-\log(T_{max}-s)\bigg{|}_{\tau_2}^t
        =\log\frac{T_{max}-\tau_2}{T_{max}-t},\quad\forall\,t\in[\tau_2, T_{max}),\non
\end{equation}
and integrating \eqref{bu3} with respect to time, we obtain
\begin{equation}
	\log\tilde{V}(t)\leq \log\tilde{V}(\tau_2)+ (1+\eta)\widetilde{C}\alpha_0\log\frac{T_{max}-\tau_2}{T_{max}-t},\quad\forall\,t\in[\tau_2, T_{max}).\non
\end{equation}
Hence,
\begin{align}
	\tilde{V}(t)
&\leq \tilde{V}(\tau_2)\bigg(\frac{T_{max}-\tau_2}{T_{max}-t}\bigg)^{(1+\eta)\widetilde{C}\alpha_0}\non\\
&\leq \tilde{V}(\tau_2)T_{max}^{(1+\eta)\widetilde{C}\alpha_0} 2^{-(1+\eta)\widetilde{C}\alpha_0}(T_{max}-t)^{-(1+\eta)\widetilde{C}\alpha_0}\non\\\
& :=M(T_{max}-t)^{-(1+\eta)\widetilde{C}\alpha_0},
 \label{bu4}
\end{align}
with $M=\tilde{V}(\tau_2)T_{max}^{(1+\eta)\widetilde{C}\alpha_0} 2^{-(1+\eta)\widetilde{C}\alpha_0}$.

Multiplying the first equation of \eqref{chemo1} by $n$ and integrating over $\Omega$, we have
\begin{align}
	&\frac{1}{2}\frac{d}{dt}\int_\Omega n^2dx +\int_\Omega |\nabla n|^2dx\non\\
    &\quad =\chi\int_\Omega n\nabla c\cdot\nabla ndx\nonumber\\
	&\quad \leq\frac{1}{2}\|\nabla n\|_{L^2(\Omega)}^2+\frac{\chi^2}{2}\int_\Omega n^2|\nabla c|^2dx\nonumber\\
	&\quad \leq\frac{1}{2}\|\nabla n\|_{L^2(\Omega)}^2+\frac{\chi^2}{2}\bigg(\int_\Omega|\nabla c|^4dx\bigg)^{\frac{1}{2}}\bigg(\int_\Omega n^4dx\bigg)^{\frac12}\nonumber\\
	&\quad \leq\frac{1}{2}\|\nabla n\|_{L^2(\Omega)}^2+\frac{\chi^2}{2}\|n\|^{\frac12}_{L^\infty(\Omega)}\bigg(\int_\Omega|\nabla c|^4dx\bigg)^{\frac{1}{2}}\bigg(\int_\Omega n^3dx\bigg)^{\frac12}.\non
\end{align}
Due to the Sobolev embedding theorem ($d=3$), we deduce that
\begin{align}
	\left(\int_\Omega n^3dx\right)^\frac13&=\|\sqrt{n}\|_{L^6(\Omega)}^2\non\\
     &\leq C(\|\nabla \sqrt{n}\|_{L^2(\Omega)}^2+\|\sqrt{n}\|_{L^2(\Omega)}^2)\non\\
     &\leq C(\|\nabla \sqrt{n}\|_{L^2(\Omega)}^2+\|n_0\|_{L^1(\Omega)})\non\\
     &\leq CM(T_{max}-t)^{-(1+\eta)\widetilde{C}\alpha_0}.\non
\end{align}
Therefore, it follows from \eqref{bu1}, \eqref{bu4} that
\begin{align}
	\frac{d}{dt}\int_\Omega n^2 dx
    &\leq\chi^2\|n\|^{\frac12}_{L^\infty(\Omega)}\bigg(\int_\Omega|\nabla c|^4dx\bigg)^{\frac{1}{2}}\bigg(\int_\Omega n^3dx\bigg)^{\frac12}\non\\
	&\leq CM^2\chi^2(T_{max}-t)^{-\frac12-2(1+\eta)\widetilde{C}\alpha_0}.\non
\end{align}
Thus, using the fact $-\frac12-2(1+\eta)\widetilde{C}\alpha_0>-1$ due to the choice of $\alpha_0$ and integrating the above inequality with respect to time yield that for any $t\in[\tau_2, T_{max})$, it holds
\begin{align}
	\|n(t)\|_{L^2(\Omega)}^2
    &\leq \|n(\tau_2)\|_{L^2(\Omega)}^2+CM^2\chi^2\int_{\tau_2}^t(T_{max}-s)^{-\frac12-2(1+\eta)\widetilde{C}\alpha_0}ds\non\\
    &= \|n(\tau_2)\|_{L^2(\Omega)}^2 +CM^2\chi^2\left(\frac12-2(1+\eta)\widetilde{C}\alpha_0\right)^{-1}(T_{max}-s)^{\frac12-2(1+\eta)\widetilde{C}\alpha_0}\bigg{|}^{\tau_2}_t\non\\
	&\leq \|n(\tau_2)\|_{L^2(\Omega)}^2 +CM^2\chi^2\left(\frac12-2(1+\eta)\widetilde{C}\alpha_0\right)^{-1}(T_{max}-\tau_2)^{\frac12-2(1+\eta)\widetilde{C}\alpha_0}.\non
\end{align}
As a consequence, we can conclude that $\|n\|_{L^\infty(0,T_{max};\, L^2(\Omega))}$ is bounded.
Recalling Corollary \ref{cor42} (taking $\e=\frac12$ in \eqref{subcri}), we have thus proved that the local classical solution $(n,c)$ can be extended beyond $T_{max}$, which leads to a contradiction with the fact that $T_{max}<+\infty$ is the maximal existing time.

Hence, the lower bound \eqref{lowblow} in Theorem \ref{mr2} must hold for a local classical solution that blows up in finite time.
This completes the proof. $\square$

\begin{remark}
Recalling Corollary \ref{smallini} and \cite{Tao}, we see that there exist a constant $\sigma>0$ possibly depending on $\Omega$ and $\chi$, if $\|c_0\|_{L^\infty(\Omega)}\leq \sigma$, then the local classical solution $(n,c)$ will not blow up in finite time, i.e., $T_{max}=+\infty$.
Thus, Theorem \ref{mr2} applies to the regime $\|c_0\|_{L^\infty(\Omega)}>\sigma$, where finite blow-up may happen.
In this case, we have
\begin{equation}
	\widetilde{C}> 3\big[2(2+\sqrt{3})\big]^{\frac{2}{3}}C_3^{\frac{1}{3}}\sigma^{\frac{4}{3}},\non
\end{equation}
and it follows from \eqref{kkk3}, \eqref{bula} that the admissible positive constant $\alpha$ in \eqref{lowblow} has an upper bound depending only on $\Omega$ and $\chi$
$$\alpha<\frac{1}{12} \big[2(2+\sqrt{3})\big]^{-\frac{2}{3}}C_3^{-\frac{1}{3}}\sigma^{-\frac{4}{3}}.$$
Moreover, we see that $\alpha\sim \|c_0\|_{L^\infty(\Omega)}^{-\frac43}$.
This implies that the larger $\|c_0\|_{L^\infty(\Omega)}$ is, the smaller $\alpha$ should be, and in this case, if
  $\limsup\limits_{t\nearrow T_{max}^-} (T_{max}-t) \|n(t)\|_{L^\infty(\Omega)} <\alpha$, then the local classical solution $(n,c)$ can be extended beyond $T_{max}$.
\end{remark}

\subsection{Proof of Theorem \ref{mr3}}
Theorem \ref{mr3} involves the blow-up rate for the local $L^\infty$-norm of $n$.
Its proof can be carried out following the argument in \cite[Section 3]{MizoSoup} with some necessary modifications.
Although here we only treat the case $d=3$, the conclusion of Theorem \ref{mr3} indeed does not relate to the spatial dimension \cite{MizoSoup}.

We sketch the proof as follows.

First, we quote the following lemma for the inhomogeneous linear heat equation (see \cite[Lemma 3.1]{MizoSoup}):
\begin{lemma}\label{lemheat}
Let $c$ be a classical solution of
\beq
 \begin{cases}\label{heat}
&c_t- \Delta c=- \lambda c+g, \quad (x,t)\in \Omega\times(0,T),\\
&\frac{\partial c}{\partial \nu}=0,\qquad\qquad \quad \ (x,t)\in \partial\Omega\times(0,T),\\
&c|_{t=0}=c_0(x),\qquad\quad  x\in \Omega.
\end{cases}
\eeq
with $g\in L_{loc}^\infty(0,T;L^\infty(\Omega))$, $\lambda\geq 0$. For any $\mu>\frac12$, there exists a constant $C^*=C^*(\Omega, \mu, \lambda)>0$ such that the following holds. Let $x^*\in \overline{\Omega}$. $t_0\in (0,T)$, $\rho, \varepsilon, \kappa>0$ and assume that
$$|g(x,t)|\leq \varepsilon(T-t)^{-\mu}, \quad \forall\,(x,t)\in \Omega_{x^*,\rho}\times(t_0,T),$$
$$\|c(t)\|_{L^1(\Omega_{x^*,\rho})}\leq \kappa, \quad \forall\, t\in (t_0,T).$$
Then for any $\tilde{\rho}\in (0,\rho)$, there exists $K>0$ such that the function $c(x,t)$ satisfies
\begin{equation}
|\nabla c(x,t)|\leq C^*\varepsilon(T-t)^{-\mu+\frac12}+K,\quad \forall\, (x,t)\in \Omega_{x^*,\tilde{\rho}}\times(t_0,T).\non
\end{equation}
\end{lemma}

Let $x^*\in \overline{\Omega}$, $t_0\in (0,T_{max})$, $\rho>0$. Besides, we assume that
\begin{equation}
	n(x,t)\leq \varepsilon (T_{max}-t)^{-1},\quad \forall\, (x,t)\in\Omega_{x^*,\rho}\times(t_0,T_{max}).\non
\end{equation}

We first consider the case $x^*\in\Omega$. Denote $\delta=\min\{\rho,\text{dist}(x^*,\partial\Omega)\}$.
We observe that the first equation of system \eqref{chemo1} for $n$ is the same as the one in the classical Keller--Segel model considered in \cite{MizoSoup} with $m=1$. Besides, in our present case, $\|c\|_{L^\infty(0, T_{max};\, L^\infty(\Omega))}$ is bounded.
As a sequence, letting $g=nc$ and $\lambda=0$ in \eqref{heat}, the conclusion of Lemma \ref{lemheat} holds true for our problem as well.
Then by the same argument as for \cite[(3.19)]{MizoSoup} with $m=1$, we can prove that for each $p>1$, there exists $C_{10}$, $K_1>0$ such that
\begin{equation}
  \label{lcnon1}
	\int_{B_{\frac{\delta}{2}}(x^*)} n^{p+1}(x,t) dx\leq K_1(T_{max}-t)^{-C_{10}\varepsilon^{2}}.
\end{equation}
Next, we claim that
\begin{equation}
	\label{lcnon2}
|\nabla c(x_0,t)|\leq K_2,\quad\forall\, x_0\in B_{\frac{\delta}{8}}(x^*)\;\text{and}\; t\in [t_0,T_{max}).
\end{equation}
To this end, taking a smooth cut-off function $\varphi\in C^2(\mathbb{R}^3)$, $0\leq \varphi\leq 1$ such that $\varphi(x)=1$ for $x\in B_{\frac{\delta}{4}}(x^*)$ and $\varphi(x)=0$ for $x\in\mathbb{R}^3\setminus B_{\frac{\delta}{2}}(x^*)$.
Put $\tilde{c}(x,t)=c(x,t)\varphi(x)$ for $\Omega\times(0,T_{max})$. Then the variable $\tilde{c}$ satisfies
\begin{equation}
	\tilde{c}_t=\Delta\tilde{c}-(nc\varphi+2\nabla c\cdot \nabla \varphi+c\Delta\varphi), \quad \mathrm{in} \ \Omega\times(0, T_{max}).\non
\end{equation}
 Let $G=G(x,y; t)$ and $(S(t))_{t\geq 0}$ denote the kernel and the semigroup associated with the Laplacian $\Delta$ with Neumann boundary conditions, respectively. Pick $x_0\in B_{\frac{\delta}{8}}(x^*)$ and $t\in[t_0, T_{max})$.
Then $\nabla \tilde{c}(x_0,t)$ can be represented as follows:
\begin{equation}
	\nabla \tilde{c}(x_0,t)=J_0(x_0,t)-J_1(x_0,t)-2J_2(x_0,t)-J_3(x_0,t),\non
\end{equation}
where
\begin{equation}
	J_0(x_0,t)=\nabla u(x_0,t)\quad\text{with}\quad u(\cdot, t)=S(t-t_0)\tilde{c}(t_0),\non
\end{equation}
and
\begin{align}
	J_1(x_0,t)&=\int_{t_0}^t\int_{\Omega}\nabla_x G(x_0,y;t-s)[n(y,s)c(y,s)\varphi(y)]dyds,\non\\
	J_2(x_0,t)&=\int_{t_0}^t\int_{\Omega}\nabla_x G(x_0,y;t-s)[\nabla c(y,s)\cdot\nabla \varphi(y)]dyds,\non\\
	J_3(x_0,t)&=\int_{t_0}^t\int_{\Omega}\nabla_x G(x_0,y;t-s)[c(y,s)\Delta \varphi(y)]dyds.\non
\end{align}
The estimates for $J_0(x_0,t)$, $J_2(x_0,t)$ and $J_3(x_0,t)$ can be obtained in the same way as in the proof of \cite[Lemma 3.1]{MizoSoup}.
It remains to control the term $J_1(x_0,t)$. We note that by \eqref{lcnon1} and in particular, the boundedness of $\|c\|_{L^\infty(0,T_{max};\ L^\infty(\Omega))}$, it holds
\begin{align}
	|J_1(x_0,t)|&\leq C\int_{t_0}^t(t-s)^{-2}\int_\Omega n(y,s)c(y,s)\varphi(y)\exp\bigg(-\frac{C|x-y|^2}{4(t-s)}\bigg)dyds\nonumber\\
	&\leq C\int_{t_0}^t(t-s)^{\frac{3}{2q}-2}
          \times\bigg[(t-s)^{-\frac{3}{2}}\int_\Omega\exp\bigg(-\frac{Cq|x-y|^2}{4(t-s)}\bigg)dy\bigg]^{\frac{1}{q}}\nonumber\\
	&\qquad\quad  \times\bigg(\int_\Omega [n(y,s)c(y,s)\varphi(y)]^p dy\bigg)^{\frac{1}{p}}ds\nonumber\\
	&\leq C\|c\|_{L^\infty(0,T_{max};\, L^\infty(\Omega))} \int_{t_0}^t(t-s)^{\frac{3}{2q}-2}\bigg[K_1(T_{max}-t)^{-C_{10}\varepsilon^2}\bigg]^{\frac{1}{p}}ds\nonumber\\
	&\leq C,\non
\end{align}
for $x_0\in B_{\frac{\delta}{8}}(x^*)$ and $t\in[t_0, T_{max})$, provided that $\varepsilon$, $p$ and $q$ satisfy
\begin{equation}
	1-\frac{3}{2q}+\frac{C_{10}\varepsilon^2}{p}<0, \quad \frac{1}{p}+\frac{1}{q}=1, \quad 1\leq q<\frac32.\non
\end{equation}
Thus, the conclusion \eqref{lcnon2} follows. Then due to Step 3 in the proof of \cite[Theorem 1.1]{MizoSoup}, we can further prove
\begin{equation}
	n(x_0,t)\leq K_3, \quad\text{for}\ x_0\in B_{\frac{\delta}{32}}(x^*) \;\mathrm{and} \;t\in[t_0, T_{max}),\non
\end{equation}
which immediately yields that $x^*$ is not a blow-up point.

At last, the case of $x^*\in\partial\Omega$ can be treated in the same way as in \cite{MizoSoup}, with corresponding modifications in the proof like those indicated above.

The proof of Theorem \ref{mr3} is complete. $\square$

\section*{Acknowledgments}
J. Jiang is partially supported by National Natural Science Foundation of China (NNSFC) under the grant No. 11201468. H. Wu is partially supported by NNSFC under the grant No. 11371098, 11631011 and Shanghai Center for Mathematical Sciences at Fudan University. S. Zheng is partially supported by NNSFC under the grant No. 11131005.


\end{document}